\documentclass[11pt]{article}
\RequirePackage{amsthm,amsmath,amsfonts,amssymb}
\usepackage[authoryear]{natbib}
\usepackage{amsfonts}
\usepackage{bbm}
\usepackage{mathrsfs}
\usepackage{subfigure}
\usepackage[dvipdfmx]{graphicx} 
\usepackage{titlesec}
\usepackage{lipsum}
\usepackage{enumerate}
\usepackage{color}
\usepackage{amssymb}
\usepackage{amsmath,amssymb,amsfonts,amsthm}
\usepackage{bm}
\usepackage{here}
\usepackage{multicol}
\usepackage{epsf}
\usepackage{authblk}

\makeatletter
\renewenvironment{proof}[1][\proofname]{\par
  \normalfont
  \topsep6\p@\@plus6\p@ \trivlist
  \item[\hskip\labelsep{\bfseries #1}\@addpunct{\bfseries.}]\ignorespaces
}{
  \endtrivlist
}
\renewcommand{\proofname}{Proof}
\makeatother
\begin{document}
\newcommand{\field}[1]{\mathbb{#1}} 
\newcommand{\ind}{\mathbbm{1}}
\newcommand{\C}{\mathbb{C}}
\newcommand{\D}{\,\mathscr{D}}
\newcommand{\E}{\,\mathcal{E}}
\newcommand{\F}{\,\mathscr{F}}
\newcommand{\I}{\,\mathrm{i}}
\newcommand{\N}{\field{N}}
\newcommand{\T}{\,\mathrm{T}}
\newcommand{\Ls}{\,\mathscr{L}}
\newcommand{\ve}{\,\mathrm{vec}}
\newcommand{\var}{\,\mathrm{var}}
\newcommand{\cov}{\,\mathrm{Cov}}
\newcommand{\vech}{\,\mathrm{vech}}
\newcommand*\dif{\mathop{}\!\mathrm{d}}
\newcommand{\difs}{\mathrm{d}}
\newcommand\mi{\mathrm{i}}
\newcommand\me{\mathrm{e}}
\newcommand{\R}{\field{R}}

\def\poi{\textsc{Poisson}}
\newcommand{\es}{\hat{f}_n}
\newcommand{\ess}{\hat{f}_{\flr{ns}}}
\newcommand{\p}[1]{\frac{\partial}{\partial#1}}
\newcommand{\pp}[1]{\frac{\partial^2}{\partial#1\partial#1^{\top}}}
\newcommand{\para}{\theta}
\providecommand{\bv}{\mathbb{V}}
\providecommand{\bu}{\mathbb{U}}
\providecommand{\bt}{\mathbb{T}}
\newcommand{\flr}[1]{\lfloor#1\rfloor}
\newcommand{\ba}{B_m}
\newcommand{\bxi}{\bar{\xi}_n}
\newcommand{\sgn}{{\rm sgn \,}}
\newcommand{\rint}{\int^{\infty}_{-\infty}}

\newcommand{\skakko}[1]{\left(#1\right)}
\newcommand{\mkakko}[1]{\left\{#1\right\}}
\newcommand{\lkakko}[1]{\left[#1\right]}

\newcommand{\Pj}{\mathbb{P}}
\newcommand{\Z}{\field{Z}}
\newcommand{\Zo}{\field{Z}_0}

\newcommand{\abs}[1]{\lvert#1\rvert}
\newcommand{\ct}[1]{\langle#1\rangle}
\newcommand{\inp}[2]{\langle#1,#2\rangle}
\newcommand{\norm}[1]{\lVert#1 \rVert}
\newcommand{\Bnorm}[1]{\Bigl\lVert#1\Bigr  \rVert}
\newcommand{\Babs}[1]{\Bigl \lvert#1\Bigr \rvert} 
\newcommand{\ep}{\epsilon} 
\newcommand{\sumn}[1][i]{\sum_{#1 = 1}^T}
\newcommand{\tsum}[2][i]{\sum_{#1 = -#2}^{#2}}
\providecommand{\abs}[1]{\lvert#1\rvert}
\providecommand{\Babs}[1]{\Bigl \lvert#1\Bigr \rvert} 

\newcommand{\uint}{\int^{1}_{0}}
\newcommand{\freqint}{\int^{\pi}_{-\pi}}
\newcommand{\li}[1]{\mathfrak{L}(S_{#1})}

\newcommand{\cum}{{\rm cum}}

\newcommand{\xt}{\bm{X}_{t, T}}
\newcommand{\yt}{\bm{Y}_{t, T}}
\newcommand{\zt}{\bm{Z}_{t, T}}
\newcommand{\gcu}{{\rm GC}^{2 \to 1}(u)}

\newcommand{\btheta}{\bm{\theta}}
\newcommand{\bbeta}{\bm{\eta}}
\newcommand{\bzeta}{\bm{\zeta}}
\newcommand{\bzero}{\bm{0}}
\newcommand{\bI}{\bm{I}}
\newcommand{\bd}{\bm{d}}
\newcommand{\bx}[1]{\bm{X}_{#1, T}}
\newcommand{\be}{\bm{\ep}}
\newcommand{\bp}{\bm{\phi}}

\newcommand{\act}{A_T^{\circ}}
\newcommand{\ac}{A^{\circ}}

\newcommand{\dlim}{\xrightarrow{d}}
\newcommand{\plim}{\rightarrow_{P}}

\newcommand{\ls}{\mathcal{S}}
\newcommand{\cs}{\mathcal{C}}

\providecommand{\ttr}[1]{\textcolor{red}{ #1}}
\providecommand{\ttb}[1]{\textcolor{blue}{ #1}}
\providecommand{\ttg}[1]{\textcolor{green}{ #1}}
\providecommand{\tty}[1]{\textcolor{yellow}{ #1}}
\providecommand{\tto}[1]{\textcolor{orange}{ #1}}
\providecommand{\ttp}[1]{\textcolor{purple}{ #1}}

\newcommand{\sign}{\mathop{\rm sign}}
\newcommand{\conv}{\mathop{\rm conv}}
\newcommand{\argmax}{\mathop{\rm arg~max}\limits}
\newcommand{\argmin}{\mathop{\rm arg~min}\limits}
\newcommand{\argsup}{\mathop{\rm arg~sup}\limits}
\newcommand{\arginf}{\mathop{\rm arg~inf}\limits}
\newcommand{\diag}{\mathop{\rm diag}}
\newcommand{\minimize}{\mathop{\rm minimize}\limits}
\newcommand{\maximize}{\mathop{\rm maximize}\limits}
\newcommand{\tr}{\mathop{\rm tr}}
\newcommand{\Cum}{\mathop{\rm Cum}\nolimits}
\newcommand{\Var}{\mathop{\rm Var}\nolimits}
\newcommand{\Cov}{\mathop{\rm Cov}\nolimits}

\numberwithin{equation}{section}
\theoremstyle{plain}
\newtheorem{thm}{Theorem}[section]

\newtheorem{lem}[thm]{Lemma}
\newtheorem{prop}[thm]{Proposition}
\theoremstyle{definition}
\newtheorem{defi}{Definition}[section]
\newtheorem{assumption}[defi]{Assumption}
\newtheorem{cor}[thm]{Corollary}
\newtheorem{rem}[thm]{Remark}
\newtheorem{eg}{Example}

\title{Sparse principal component analysis for high-dimensional stationary time series}
\author[1]{Kou Fujimori}
\author[2]{Yuichi Goto}
\author[3]{Yan Liu}
\author[2]{Masanobu Taniguchi}
\affil[1]{Department of Economics,
Faculty of Economics and Law,
Shinshu University.}
\affil[2]{Department of Applied Mathematics,
Waseda University.}
\affil[3]{Institute for Mathematical Science,
Faculty of Science and Engineering,
Waseda University.}
\date{}
\maketitle
\begin{abstract}
We consider the sparse principal component analysis for high-dimensional stationary processes.
The standard principal component analysis performs poorly when the dimension of the process is large.
We establish the oracle inequalities for penalized principal component estimators for the processes including heavy-tailed time series.
The rate of convergence of the estimators is established.
We also elucidate the theoretical rate for choosing the tuning parameter in penalized estimators.
The performance of the sparse principal component analysis is demonstrated by numerical simulations.
The utility of the sparse principal component analysis for time series data is exemplified by the application to average temperature data.
\end{abstract}

\section{Introduction.}
The principal component analysis (PCA) has been a standard tool for multivariate data analysis.
It facilitates the understanding of the covariance matrix 
and
becomes a central method for dimension reduction and variable selection.
When the sample size is large, \cite{anderson1963} developed the asymptotic theory for principal component analysis.
A thorough investigation into the standard principal component analysis is summarized in \citet{jolliffe2002}.

The dimension $p$ of the contemporary data is often large, compared with the sample size.
\cite{johnstone2001} investigated the distribution of the largest eigenvalue when $p$ is large,
and introduced the concept of the spiked covariance matrix model.
The sparse principal component analysis, combined with variable selection techniques such as Lasso (\cite{tibshirani1996} or elastic net \cite{zh2005}), 
was introduced in \cite{zou2006}.
\cite{sh2008} considered the sparse principal component analysis via regularized low-rank matrix approximation.
\cite{johnstone2009} provided a simple algorithm for selecting a subset of coordinates with the largest sample variances with consistency even when the dimension $p$ is large.
\cite{amini2009} proposed two computational methods for recovering the support set of the leading eigenvector
in the spiked covariance model.
\cite{cai2012} derived the optimal rates of convergence for sparse covariance matrix estimation.
\cite{pj2012} proposed an augmented sparse PCA method and showed that the procedure attains near optimal rate of convergence.
\cite{birnbaum2013} studied the problem of estimating the leading eigenvector under the $l_2$-loss for independent high-dimensional Gaussian observations.
\cite{cai2013} considered both minimax and adaptive estimation of the principal subspace in the high dimensional setting.
\cite{vu2013minimax} also considered the sparse principal subspace estimation problems and established the 
optimal bounds for row subspace and 
nearly optimal for column subspace.
\cite{berthet2013} derived a minimax optimality in a finite sample analysis for sparse principal components 
of a high-dimensional covariance matrix.
In view of computation of the sparse principal 
component analysis, the several work has been 
appeared.
For example, 
\cite{ma2013} proposed the iterative thresholding method for estimation of the leading eigenvector, 
and \cite{wang2016statistical} studied the computationally efficient algorithm to estimate the principal subspace.
\cite{van2016estimation} formalized the theoretical development for sparse PCA on some local set,
which is induced to ensure the compatibility condition (See also, e.g. \cite{buhlmann2011}).
Notably, much of the above theory was developed under the setting of i.i.d.~observations.

The principal component analysis applied to dependent data has also been studied for a long time.
\cite{zhao1986} proposed a new procedure for detection of signals based on eigenvalues of covariance matrix.
\cite{tk1987} derived the asymptotic distributions of eigenvalues of the sample covariance matrix from Gaussian stationary processes.
However, all these developments are restricted to the case when the dimension $p$ is finite,
i.e., the PCA for multivariate stationary processes.
More details of analyses for multivariate stationary processes can be found in \cite{taniguchikakizawa2000}.
The limiting distribution of sample covariance matrix for large-dimensional linear models was derived in \cite{jin2009}.
The Mar\v{c}enko-Pastur theorem for time series was obtained in an explicit way in \cite{yao2012note}.
The high-dimensional covariance estimation for dependent data with some regularization techniques was introduced in \cite{pourahmadi2013}.
The theoretical development for regularized estimation in sparse high-dimensional time series models
was considered in \cite{basu2015regularized}.
They showed that a restricted eigenvalue condition holds with high probability.
Motivated by this work,
\cite{wong2020lasso} established the consistency of Lasso for some sparse non-Gaussian and nonlinear time series.

In this paper, we consider the sparse principal component analysis for high-dimensional stationary time series.
Especially, we established the oracle inequalities for the Lasso-type PCA estimator for both
$\alpha$-mixing Gaussian process and $\beta$-mixing sub-Weibull process.
In addition, we also derived the oracle inequality for $l_0$-penalized estimators for comparison.
The finite sample performance is illustrated by some numerical simulations.

\subsection{Notation.}
For a vector $\bm{v} \in \mathbb{R}^p$, we defined the $l_r$-norm 
$\|\bm{v}\|_r$ as
$\|\bm{v}\|_r = \left(
\sum_{i=1}^p |v_i|^r
\right)^{1/r}$ for $r \in (0, \infty)$.
Also, let $\|\bm{v}\|_0$ and $\|\bm{v}\|_\infty$ be
$\|\bm{v}\|_0
= \sum_{i=1}^p \mathbbm{1}_{\{|v_i| >0\}}$ and 
$\|\bm{v}\|_\infty
= \max_{1 \leq i \leq p} |v_i|$, respectively.

For a $p \times p$ matrix $A$, the operator norm $\|A\|_r$ is defined as
\[
\|A\|_r
:= \sup_{\|v\|_r=1} \|A \bm{v}\|_r,\quad
r \in (0, \infty].
\]
Moreover, the ``max" norm of the matrix $A$ is
$\|A\|_{\max}
:= \max_{i,j} |A_{ij}|$.
For a vector $\bm{v} \in \mathbb{R}^p$, and an index set $T\subset \{1,2,\ldots,p\}$, 
we denote by $\bm{v}_T$ the $|T|$-dimensional sub-vector of $\bm{v}$ restricted by the index set $T$,
where $|T|$ is the number of elements of the set $T$. 

The rest of the paper is organized as follows.
In Section \ref{sec:2}, we provide the fundamental settings for the sparse principal component analysis for stationary processes.
Theoretical results of the Lasso-type principal component analysis for Gaussian processes and heavy-tailed processes are discussed in Sections \ref{sec:3} and \ref{sec:4}, respectively.
In addition, the $l_0$-penalized estimation is discussed in Section \ref{sec:5}.
Section \ref{sec:6} gives several simulation results to demonstrate the finite sample performance of sparse principal component analyses.
The rigorous proofs and technical results are relegated to Section \ref{sec:7}.

\section{Preliminaries.} \label{sec:2}
\subsection{Model setups.}\label{subsec:2.1}
Let $\{\bm{X}_t\}_{t \in \mathbb{Z}}$
be an $\mathbb{R}^p$-valued, strictly stationary, and centered time series on a probability space $(\Omega, \mathcal{F}, P)$.
Suppose the observation stretch $(\bm{X}_1,\ldots,\bm{X}_n),\ n \in \mathbb{N}$, is available.
Consider the following $p \times p$ matrices
\[
\Sigma_0 = E[\bm{X}_t \bm{X}_t^\top],
\qquad \qquad
\hat{\Sigma}_n
= \frac{1}{n}\sum_{t=1}^n \bm{X}_t \bm{X}_t^{\top}.
\]
Let $\bm{q}^0$ be the first principal component  corresponding to the largest eigenvalue 
$\phi_{\max}^2 := \Lambda_{\max}(\Sigma_0)$ of $\Sigma_0$, so that $\bm{q}^0$ is normalized as $\|\bm{q}^0\|_2=1$.
The parameter of interest is
\[
\bm{\beta}^0
= \phi_{\max} \bm{q}^0,
\]
which is a solution to the following optimization 
problem
\[
\bm{\beta}^0
= \arg \min_{\bm{\beta} \in \mathbb{R}^p} \frac{1}{4}
\|
\Sigma_0 - \bm{\beta} \bm{\beta}^\top
\|_F^2,
\]
where $\|\cdot\|_F$ is the Frobenius norm.
In other words, it holds that
\[
\Sigma_0 \bm{\beta}^0 = \phi_{\max}^2 \bm{\beta}^0 =  \|\bm{\beta}^0\|_2^2 \bm{\beta}^0.
\]

Our primary interest is the sparse principal component estimation.
Let $S$ be $S = \{j : \bm{\beta}^0_j \not=0\}$.
Specifically, $\bm{\beta}^0$ is supposed to be \emph{$s_0$-sparse}, $i.e.,$ 
$|S| = s_0$.
A typical motivating example is given as follows.

\begin{eg}
Consider the stationary process $\{\bm{X}_t\}$ taking the form of VAR model, i.e.,
\begin{equation} \label{eq:2.1}
\bm{X}_t = A \bm{X}_{t - 1} + \bm{\epsilon}_t, \qquad \bm{\epsilon_t} \sim {\rm i.i.d.}(\bm{0}, I_p),
\end{equation}
where $A$ is a $p \times p$ deterministic matrix with the decomposition
\[
A = \sum_{j = 1}^p \nu_j \bm{p}_j \bm{p}_j^{\top}
\]
such that $1 > \nu_1 \geq \cdots \geq \nu_p \geq 0$ are the eigenvalues of $A$, 
with $\bm{p}_j$ the associated eigenvectors.
Suppose the eigenvector $\bm{p}_1$ is $s_0$-sparse.
By the holomorphic functional calculus, we have
\[
\Sigma_0 = \sum_{j = 1}^p \frac{1}{1 - \nu_j^2} \bm{p}_j \bm{p}_j^{\top},
\]
which shows that the first principal component of $\Sigma_0$ is also $s_0$-sparse.
\end{eg}

In this paper, we consider the following penalized PCA estimators. 
In such a high-dimensional setting, 
the sparse estimation is essential for
variable selection, which facilitates the 
interpretation of features in the dataset.
\begin{defi}\label{Lasso estimator}
The following estimator for $\bm{\beta}^0$ is defined as
\[
\hat{\bm{\beta}}_n
:= \arg \min_{\beta \in \mathcal{B}}
\left\{
\frac{1}{4} \|\hat{\Sigma}_n - \bm{\beta} \bm{\beta}^\top\|_F^2 + \lambda \, {\rm pen}(\bm{\beta})
\right\},\quad
\mathcal{B}:= \{\bm{\beta}: \|\bm{\beta}-\bm{\beta}^0\|_2 \leq \eta\},
\]
where $\lambda \geq 0$ is a tuning parameter, ${\rm pen}(\cdot)$ is some penalty function,
and $\eta >0$
is a suitable constant. 
The following estimators are focused on in this paper.
The $l_1$-penalized estimator $\hat{\bm{\beta}}_n^1$ is defined as
\begin{equation}\label{lasso def}
\hat{\bm{\beta}}_n^1
:= \arg \min_{\beta \in \mathcal{B}}
\left\{
\frac{1}{4} \|\hat{\Sigma}_n - \bm{\beta} \bm{\beta}^\top\|_F^2 + \lambda_1 \|\bm{\beta}\|_1 
\right\},
\end{equation}
where $\lambda = \lambda_1$ and ${\rm pen}(\bm{\beta}) = \|\bm{\beta}\|_1 $.
The estimator $\hat{\bm{\beta}}_n^1$ is also referred to as the \emph{Lasso-type estimator}.
The $l_0$-penalized estimator $\hat{\bm{\beta}}_n^0$ is defined as
\begin{equation}\label{l0 def}
\hat{\bm{\beta}}_n^0
:= \arg \min_{\beta \in \mathcal{B}}
\left\{
\frac{1}{4} \|\hat{\Sigma}_n - \bm{\beta} \bm{\beta}^\top\|_F^2 + \lambda_0 \|\bm{\beta}\|_0
\right\},
\end{equation}
where $\lambda = \lambda_0$ and ${\rm pen}(\bm{\beta}) = \|\bm{\beta}\|_0$.
\end{defi}

We establish the error bound of the estimator $\hat{\bm{\beta}}_n$ for high-dimensional time series.
In the following, we list the definition of mixing coefficients for stochastic processes.
\begin{defi}\label{mixing coefficient}
For a stationary process $\{\bm{X}_t\}_{t \in \mathbb{Z}}$, we define the following 
quantities.
\begin{itemize}
\item[{\rm (i)}]
The $\alpha$-mixing coefficients for $\{\bm{X}_t\}_{t \in \mathbb{Z}}$ is defined as
\begin{eqnarray*}
\alpha(l) 
&:=& \sup\{
|P(A \cap B)-P(A)P(B)| : \\
&&A \in \sigma(\bm{X}_s,\ s \leq t),\ 
B \in \sigma(\bm{X}_s,\ s \geq t+l)
\ \mbox{for all}\ t \in \mathbb{Z}
\},\quad
l \in \mathbb{Z}.
\end{eqnarray*}
The process $\{\bm{X}_t\}_{t \in \mathbb{Z}}$ is called \emph{$\alpha$-mixing} if 
$\alpha(l) \to 0$ as $l \to \infty$.
\item[{\rm (ii)}]
The $\rho$-mixing coefficients for 
$\{\bm{X}_t\}_{t \in \mathbb{Z}}$ is defined as
\begin{eqnarray*}
\rho(l) 
&:=& \sup\{
\Cov(f(\bm{X}_t), g(\bm{X}_{t + l})) : \\
&& E[f]=E[g]=0,\ E[f^2]=E[g^2]=1
\},\quad
l \in \mathbb{Z}.
\end{eqnarray*}
The process $\{\bm{X}_t\}_{t \in \mathbb{Z}}$ is called  \emph{$\rho$-mixing} if 
$\rho(l) \to 0$ as $l \to \infty$.
\item[{\rm (iii)}]
The $\beta$-mixing coefficients for $\{\bm{X}_t\}_{t \in \mathbb{Z}}$ is defined by
\begin{eqnarray*}
\beta(l) 
&:=& \sup
\frac{1}{2}\sum_{i=1}^I \sum_{j=1}^J|P(A \cap B)-P(A)P(B)|,\quad
l \in \mathbb{Z},
\end{eqnarray*}
where the supremum is over all pair of
partitions 
$\{A_i\}_{1 \leq i \leq I} \subset \sigma(\bm{X}_s,\ s \leq t)$
and $\{B_j\}_{1 \leq j \leq J} \subset \sigma(\bm{X}_s,\ s \geq t+l)$.
The process $\{\bm{X}_t\}_{t \in \mathbb{Z}}$ is called \emph{$\beta$-mixing} if 
$\beta(l) \to 0$ as $l \to \infty$.
\end{itemize}
\end{defi}
All these conditions are known as the weak dependence conditions (e.g. \cite{tikhomirov1981}),
under which the convergence rate in the central limit theorem for weakly dependent random variables is evaluated.

Now consider the spectral decomposition 
of $\Sigma_0$ as
\[
\Sigma_0 = U \Phi_0 U^\top,
\]
where $\Phi_0^2 := \diag(\phi_1^2,\ldots,\phi_p^2)$ with $\phi_1 \geq \ldots \geq \phi_p \geq 0$
is a diagonal matrix constructed by the 
eigenvalues of $\Sigma_0$, and 
$U= (\bm{u}_1,\ldots,\bm{u}_p)$ satisfies that $U U^\top= U^\top U = I_p$.
Actually, we have $\phi_{\max} = \phi_1$ and 
$\bm{q}^0 = \bm{u}_1$.
Hereafter, we assume the following 
conditions.
\begin{assumption}\label{model setups}
Suppose that the dimension $p: = p(n)$ of the process $\{\bm{X}_t\}$ satisfies
$\log p/n = o(1)$
as $n \to \infty$.
There exists a constant $\sigma> 3 \eta > 0$ such that
\[
\phi_{\max} \geq \phi_j+\sigma,\quad
j \geq 2.
\]
\end{assumption}
The second condition separates the largest eigenvalue from other eigenvalues.

\subsection{Risk functions.}\label{subsec:2.2}
We define the theoretical risk $R(\cdot)$ 
and empirical risk $R_n(\cdot)$ and their derivatives with respect to $\bm{\beta}$ as follows:
\[
R(\bm{\beta})
:= - \frac{1}{2} \bm{\beta}^\top \Sigma_0 \bm{\beta} + \frac{1}{4} \|\bm{\beta}\|_2^4,\quad
R_n(\bm{\beta})
:= - \frac{1}{2} \bm{\beta}^\top \hat{\Sigma}_n \bm{\beta} + \frac{1}{4} \|\bm{\beta}\|_2^4,
\]
\[
\dot{R}(\bm{\beta})
= - \Sigma_0 \bm{\beta} + \|\bm{\beta}\|_2^2 \bm{\beta},\quad
\dot{R}_n(\bm{\beta})
= - \hat{\Sigma}_n \bm{\beta} + \|\bm{\beta}\|_2^2 \bm{\beta},
\]
and
\[
\ddot{R}(\bm{\beta})
:= -\Sigma_0 + \|\bm{\beta}\|_2^2 I_p + 2 \bm{\beta} \bm{\beta}^\top,\quad
\ddot{R}_n(\bm{\beta})
:= -\hat{\Sigma}_n + \|\bm{\beta}\|_2^2 I_p + 2 \bm{\beta} \bm{\beta}^\top.
\]
Then, it holds that 
\[
\bm{\beta}^0 = \arg \min_{\bm{\beta} \in \R^p} R(\bm{\beta}),\quad
\hat{\bm{\beta}}_n
= \arg \min_{\bm{\beta} \in \mathcal{B}} \left\{
R_n(\bm{\beta}) + \lambda \, {\rm pen}(\bm{\beta})
\right\}.
\]
For brevity, we denote $W_n = \hat{\Sigma}_n - \Sigma_0$.
Note that 
\[
\dot{R}_n(\bm{\beta})-\dot{R}(\bm{\beta})
= -W_n\bm{\beta}
\]
and that 
\[
\ddot{R}_n(\bm{\beta})-\ddot{R}(\bm{\beta})
= -W_n.
\]
The following result holds for the risk function.
\begin{prop}[Lemma 12.7, \cite{van2016estimation}]\label{theoretical risk convex}
Under Assumption \ref{model setups}, 
$\ddot{R}(\bm{\beta})$ is positive definite on $\bm{\beta} \in \mathcal{B}$.
\end{prop}

By Proposition \ref{theoretical risk convex},
the theoretical risk function $R(\cdot)$ is shown to be strictly convex on $\mathcal{B}$.
This motivates us to consider the penalized PCA estimators as an optimization problem.
We always assume that the bound $\eta$ of $\mathcal{B}$ satisfies $\eta < \sigma/3$ later on in this paper.

\section{Lasso-type estimator for $\alpha$-mixing Gaussian process.}\label{sec:3}
We first deal with the Lasso-type estimator 
$\hat{\bm{\beta}}_n^1$ for time series 
satisfying the following condition.
\begin{assumption}\label{Gaussian}
The process $\{\bm{X}_t\}_{t \in \mathbb{Z}}$ is a zero mean and 
$\alpha$-mixing Gaussian stationary process.
\end{assumption}


An $\alpha$-mixing Gaussian process is also $\rho$-mixing. 
Let $\tau_n$ be a sequence such that
\[
\tau_n := \frac{\sigma-3 \eta}{\phi_{\max} \sum_{l=0}^n \rho(l)},
\]
and $\zeta_n$ a sequence such that 
\[
\zeta_n := \sqrt{\frac{2(b+1)\log p}{\tilde{c}n}},
\]
where $b>0$ is a free parameter and 
$\tilde{c}>0$ is a constant.
The sequence $\tau_n$ is used to 
ensure the convexity of the empirical risk,
and $\zeta_n$ is used to control
the error bound $\|W_n\|_{\max}$.
The oracle inequality is established as follows.
\begin{thm}\label{oracle ineq Lasso}
Suppose Assumptions \ref{model setups} and \ref{Gaussian} hold.
For every $n$ satisfying that 
$\log p /n \leq 1$, 
$0 < \xi_n < \tau_n$,
and $\lambda_1$ satisfying that
\[
\lambda_1 > \biggl(\frac{C+1}{C-1}\phi_{\max} \|\bm{\beta}^0\|_1 \biggr) \biggl(\zeta_n \sum_{l=0}^n \rho(l)\biggr)
\]
with some universal constant $C > 1$,
it holds that 
\begin{multline*}
P\left(
\|\hat{\bm{\beta}}_n^1-\bm{\beta}^0\|_1
\leq \frac{2(C+1)^2  s_0 \lambda_1}{\sigma-3\eta- \phi_{\max} \, \xi_n \sum_{l=0}^n \rho(l)}
\right) \\
\geq 1-\exp(-2 b \log p) - 2 \exp\bigl(-\tilde{c}n\min(\xi_n, \xi_n^2)\bigr).
\end{multline*}
\end{thm}

Noting that $b >0$ is a free parameter, we can 
take it arbitrary.
The larger $b$ implies the smaller probability in the right-hand-side of the oracle inequality, 
on the other hand for such situation, the sample size needs to be larger to ensure that the error bound of the estimator is small.
By a straightforward computation, we have
\[
\frac{\sigma-3\eta- \phi_{\max} \, \xi_n \sum_{l=0}^n \rho(l)}{\sum_{l=0}^n \rho(l)}
=
\phi_{\max} (\tau_n - \xi_n) > 0,
\]
on the domain $0 < \xi_n < \tau_n$, 
which shows the upper bound for $\|\hat{\bm{\beta}}_n^1-\bm{\beta}^0\|_1$ in parentheses is well defined.
Let $\gamma_n$ be
\begin{equation} \label{eq:3.1l}
\gamma_n =
\phi_{\max}\biggl(\zeta_n \sum_{l=0}^n \rho(l)\biggr).
\end{equation}
Since $C>1$ is a constant, the main part of the 
error bound is the order of $s_0 \lambda_1$, where $\lambda_1$ is the tuning parameter in $l_1$-penalized estimator. 
Note that Theorem \ref{oracle ineq Lasso} holds for
$\lambda_1$ larger than $\gamma_n\|\bm{\beta}^0\|_1$
up to some constant multiplication,
where $\gamma_n$ is the bound for
$\|W_n\|_{\max}=\|\hat{\Sigma}_n-\Sigma_0\|_{\max}$.
Since $\gamma_n$ depends on the mixing coefficients, the tuning parameter should be 
chosen by considering the dependence of the process.
The probability for the oracle inequality mainly depends on $p$ and $\xi_n$,
since $b$ can be taken arbitrarily. 
The sequence $\xi_n$ also depends on the mixing coefficients,
since it is bounded by $\tau_n$.
Thus, the oracle inequality of the estimator is obtained in terms of dependence residing in time series.
\begin{rem}
For a simple interpretation, we can summarize 
Theorem \ref{oracle ineq Lasso} as follows.
Let $\delta_n >0$ be a decreasing sequence,
$A_n$ and $B_n$ be sequences defined by
\[
A_n := \biggl(\frac{C+1}{C-1}\phi_{\max} \|\bm{\beta}^0\|_1 \biggr) \biggl(\zeta_n \sum_{l=0}^n \rho(l)\biggr)\quad \mbox{and} \quad
B_n := \frac{2(C+1)^2}{\sigma-3\eta- \phi_{\max} \, \xi_n \sum_{l=0}^n \rho(l)},
\]
respectively.
If $\lambda_1 > A_n$, then it holds that 
$\|\hat{\bm{\beta}}_n^1 - \bm{\beta}^0\|_1 \leq B_n s_0 \lambda_1$ with probability at least $1 - \delta_n$.
\end{rem}
Note that $\|\bm{\beta}^0\|_1$ appeared in 
the lower bound of $\lambda_1$ in Theorem 
\ref{oracle ineq Lasso} is evaluated as follows:
\[
\|\bm{\beta}^0\|_1 \leq \sqrt{s_0} \|\bm{\beta}^0\|_2 = \sqrt{s_0} \phi_{\max}.
\]
Under some additional conditions, 
we can derive the rate of convergence of the 
estimator as follows.
\begin{cor}\label{rate of convergence Lasso}
Suppose that 
$\sum_{l=0}^\infty \rho(l) < \infty$,
$\log p \to \infty$ as $n \to \infty$,
and that
\[
\lambda_1 \asymp \sqrt{s_0} \phi_{\max}^2\sqrt{\frac{\log p}{n}}.
\]
Then, under the same assumptions as 
Theorem \ref{oracle ineq Lasso}, 
the following hold true.
\[
\|\hat{\bm{\beta}}_n^1 - \bm{\beta}^0\|_1
= O_p\left(
s_0^{3/2} \phi_{\max}^2\sqrt{\frac{\log p}{n}}
\right),
\] 
and
\[
\|\hat{\bm{\beta}}_n^1 - \bm{\beta}^0\|_2^2
= O_p\left(
s_0^3 \phi_{\max}^4\frac{\log p}{n}
\right),\quad
\]
as $n \to \infty$.
\end{cor}
Note that the dimension $p$ is assumed that
$\log p/ n =o(1)$ as $n \to \infty$, which allows that $p \gg n$.
\section{Lasso-type estimator for $\beta$-mixing sub-Weibull process.} \label{sec:4}
Next, we consider the 
Lasso-type estimator $\hat{\bm{\beta}}_n^1$
for stationary processes with heavy tails.
\begin{defi}[Sub-Weibull random variables]
Let $\gamma >0$. The sub-Weibull $(\gamma)$ is defined as follows.
\begin{itemize}
\item[{\rm (i)}]
An $\mathbb{R}$-valued random variable 
$X$ is called the sub-Weibull $(\gamma)$
if it satisfies that there exists a constant $K>0$ such that
\[
\left(
E[|X|^q]
\right)^{\frac{1}{q}}
\leq K q^{\frac{1}{\gamma}}
,\quad
\forall q \geq 1 \land \gamma. 
\]
The sub-Weibull $(\gamma)$-norm
$\|\cdot\|_{\psi_\gamma}$ is defined for 
sub-Weibull $(\gamma)$ random variable 
$X$ as follows:
\[
\|X\|_{\psi_\gamma}
:= \sup_{q \geq 1} \left(
E[|X|^q]
\right)^{\frac{1}{q}} q^{- \frac{1}{\gamma}}.
\] 
\item[{\rm (ii)}]
The $\mathbb{R}^p$-valued random variable $\bm{X}$ is called
the sub-Weibull $(\gamma)$ if it satisfies that 
for every $j = 1,\ldots,p$, the $j$-th component $X_j$ of $\bm{X}$ is 
sub-Weibull $(\gamma)$.
Then, the sub-Weibull $(\gamma)$-norm
$\|\cdot\|_{\psi_\gamma}$ is defined for 
sub-Weibull $(\gamma)$ random variable 
$\bm{X}$ as follows:
\[
\|\bm{X}\|_{\psi_\gamma}
:= \sup_{v \in \mathbb{S}^{p-1}}
\|\bm{v}^\top \bm{X}\|_{\psi_{\gamma}},
\]
where $\mathbb{S}^{p-1}$ is the unit sphere
on $\mathbb{R}^p$.
\end{itemize}
\end{defi}
The sub-Gaussian random variables is sub-Weibull $(2)$; 
the sub-exponential random variable is sub-Weibull $(1)$.
Note that, for $\gamma<1$, sub-Weibull
$(\gamma)$ random variables have heavier
tail than sub-exponential and sub-Gaussian random variables.

We consider the stationary process $\{\bm{X}_t\}_{t \in \mathbb{Z}}$ satisfying the following conditions.
\begin{assumption}\label{sub-Weibull process}
\begin{itemize}
\item[(i)]
The process $\{\bm{X}_t\}_{t \in \mathbb{Z}}$ is geometrically $\beta$-mixing, $i.e.$, there exist constants $c$, 
$\gamma_1 >0$ such that 
\[
\beta(n) \leq 2 \exp\left(-c n^{\gamma_1}\right),\quad \forall n \in \mathbb{N}.
\]
\item[(ii)]
For a constant $\gamma_2>0$,
the process $\{\bm{X}_t\}_{t \in \mathbb{Z}}$ is sub-Weibull $(\gamma_2)$,
that is, 
there exists a constant $K>0$ such that,
\[
\|\bm{X}_t\|_{\psi_{\gamma_2}}
\leq K,\qquad
\forall t \in \Z.
\]
\item[(iii)]
It holds that
\[
\left(
\frac{1}{\gamma_1} + \frac{2}{\gamma_2}
\right)^{-1} < 1,
\]
where $\gamma_1$ and $\gamma_2$
are defined in (i) and (ii), respectively.
\item[(iv)]
It holds that 
\[
\frac{\log p}{n^\gamma} = o(1),\quad
n \to \infty,
\]
where 
\[
\gamma := \left(
\frac{1}{\gamma_1} + \frac{2}{\gamma_2}
\right)^{-1}.
\]
\end{itemize}
\end{assumption}

For the constant $K$ in Assumption \ref{sub-Weibull process} (ii), we defined another constant $K_2$ as
\[
K_2 := 2^{2/\gamma_2} K^2.
\]
Introducing $K_2$ can reduce terms in the oracle inequality in the following.
As a remark, for any $\bm{X}_t$, we have
\[
\|X_{tj}^2\|_{\psi_{\gamma_2}} \leq K_2, \qquad j = 1, \dots, p.
\]

Let $\tilde{b}$ be some positive constant.
Let $\tau_{1n}$ and $\tau_{2n}$ be sequences such that 
\[
\tau_{1n} := \frac{K_2 C_1^{1/\gamma}(\log n)^{1/\gamma}}{n},
\]
and
\[
\tau_{2n} := \max\left\{
\frac{K_2 C_1^{1/\gamma}(\log n p^2 + 2 \tilde{b} \log p)^{1/\gamma}}{n},
K_2 \sqrt{\frac{2 C_2 (\tilde{b} + 1)\log p}{n}}
\right\},
\]
respectively.
The sequence $\{\tau_{1n}\}$ ensures the convexity of the empirical risk,
and the sequence $\{\tau_{2n}\}$ controls the error bound for $\|W_n\|_{\max}$,
respectively.
The oracle inequality is obtained as follows.
\begin{thm}\label{oracle ineq sub-Weibull}
Suppose Assumptions \ref{model setups} and \ref{sub-Weibull process} hold.
Let $\zeta_n > \tau_{2n}$.
For $n >4$, $\tau_{1n}<\xi_n < \sigma-3 \eta$,
and $\lambda_1$ satisfying that
\[
\lambda_1 > \frac{C+1}{C-1}\zeta_n \|\bm{\beta}^0\|_1,
\]
where $C \geq 1$ is a universal constant,
it holds that 
\begin{multline*}
P\left(
\|\hat{\bm{\beta}}_n^1-\bm{\beta}^0\|_1
\leq \frac{2(C+1)^2  s_0 \lambda_1}{\sigma-3\eta-\xi_n}
\right) \\
\geq 
1- 2 \exp(-\tilde{b} \log p^2)
-2n \exp\left(-\frac{(\xi_n n)^\gamma}{K_2^\gamma C_1}\right)
- 2 \exp\left(
-\frac{\xi_n^2 n}{K_2^2 C_2}
\right),
\end{multline*}
where $\tilde{b}>0$ is a free parameter.
\end{thm}
As well as Gaussian case, the parameter $\tilde{b}>0$ is a free parameter.
Therefore, the trade off between the 
sample size and the probability can be observed. 
The condition $\tau_{1n} < \xi_n < \sigma-3\eta$ ensures that $\sigma-3\eta-\xi_n>0$.
Since $C>1$ is a universal constant, the main part of the 
error bound is still the order of $s_0 \lambda_1$, 
where $\lambda_1$ is the tuning parameter in $l_1$-penalized estimator. 
Theorem \ref{oracle ineq sub-Weibull} holds for
$\lambda_1$ larger than $\zeta_n\|\bm{\beta}^0\|_1$
up to some constant multiplication,
where $\zeta_n$ is the bound for 
$\|W_n\|_{\max}=\|\hat{\Sigma}_n-\Sigma_0\|_{\max}$.
Interestingly, 
the error bound and the tail probability do not depend on 
the mixing coefficients,
since we assume the geometric $\beta$-mixing condition here
for the process.
However, it is clear that 
the decay of the probability 
is slower than that for Gaussian case.
\begin{rem}
We can summarize
Theorem \ref{oracle ineq sub-Weibull} as follows.
Let $\delta_n >0$ be a decreasing sequence,
$A_n$ and $B_n$ be sequences defined by
\[
A_n := \frac{C+1}{C-1}\zeta_n \|\bm{\beta}^0\|_1 \quad \mbox{and} \quad
B_n := \frac{2(C+1)^2}{\sigma-3\eta-\xi_n},
\]
respectively.
If $\lambda_1 > A_n$, then it holds that 
$\|\hat{\bm{\beta}}_n^1 - \bm{\beta}^0\|_1 \leq B_n s_0 \lambda_1$ with probability at least $1 - \delta_n$.
\end{rem}

The rate of convergence of the estimator 
$\hat{\bm{\beta}}_n^1$ is 
established under some additional conditions.
\begin{cor}\label{rate of convergence sub-Weibull}
Suppose that the same assumptions as Theorem \ref{oracle ineq sub-Weibull} hold.
Assume moreover that 
\[
\xi_n \asymp \frac{(\log n)^{1/\gamma}}{n^{1-\alpha}},\quad
\alpha \in (0, 1/2),
\]
and
\[
\lambda_1  \asymp \sqrt{s_0} \phi_{\max} \sqrt{\frac{\log p}{n}}.
\]
Then, the following hold true. 
\[
\|\hat{\bm{\beta}}_n^1 - \bm{\beta}^0\|_1
= O_p\left(
s_0^{3/2} \phi_{\max} \sqrt{\frac{\log p}{n}}
\right),
\]
and
\[
\|\hat{\bm{\beta}}_n^1 - \bm{\beta}^0\|_2^2
= O_p\left(
s_0^3 \phi_{\max}^2 \frac{\log p}{n}
\right),\quad
\]
as $n \to \infty$.
\end{cor}
\begin{rem}
As far as $\phi_{\max}$ obeys the constant order, 
the rates of convergence of the estimator in both cases are $\|\hat{\bm{\beta}}_n^1 - \bm{\beta}^0\|_1 = O_p(s_0^{3/2} \sqrt{\log p/n})$.
However, for sub-Weibull case, 
we can see that the decay of the 
tail probability derived in 
Theorem \ref{oracle ineq sub-Weibull} is
slower than the corresponding result for 
Gaussian processes, which is caused by the heavy-tail property of sub-Weibull distribution when $\gamma_2  < 1$.
In addition, note that 
the mixing condition assumed in Section \ref{sec:4} is stronger than that for the Gaussian case.
\end{rem}
\section{$l_0$-penalized estimator for Gaussian process.} \label{sec:5}
In this section, we discuss the $l_0$-penalized estimator $\hat{\bm{\beta}}_n^0$ for $\alpha$-mixing Gaussian stationary processes.
The estimator $\hat{\bm{\beta}}_n^0$ is defined as
\begin{equation}\label{l0 penalty}
\hat{\bm{\beta}}_n^0
:= \arg \min_{\bm{\beta} \in \mathcal{B}}
\{R_n (\bm{\beta}) + \lambda_0 \|\bm{\beta}\|_0\},
\end{equation}
where $\lambda_0 \geq 0$ is a tuning parameter for $l_0$-penalized estimator.
Note that for the penalized estimator
(\ref{l0 penalty}), there exists a nonnegative constant $s \geq 0$
such that
\begin{equation*} 
\hat{\bm{\beta}}_n^0
:= \arg \min_{\bm{\beta} \in \mathcal{B} \cap \mathcal{B}_0(s)} R_n(\bm{\beta}),
\end{equation*}
where 
\[
\mathcal{B}_0(s)
:= \{\bm{\beta} \in \mathbb{R}^p : 
\|\bm{\beta}\|_0 \leq s\}.
\]

Let $\tilde{s} = \max\{s_0, s\}$, and $\tilde{\zeta}_n$ be a sequence such that 
\[
\tilde{\zeta}_n := \sqrt{\frac{(b+4\tilde{s}) \log p}{cn}},
\]
where $b$ and $c$ are some positive constants.
Let $\tilde{\gamma}_n$ be defined as
\begin{equation*}
\tilde{\gamma}_n =
\phi_{\max}\biggl( \tilde{\zeta}_n \sum_{l=0}^n \rho(l)\biggr).
\end{equation*}
This is parallel to the definition of $\gamma_n$ in \eqref{eq:3.1l} for Lasso-type estimator.
The oracle inequality is then derived as follows.
\begin{thm}\label{oracle l0}
Suppose that Assumptions \ref{model setups} and \ref{Gaussian} hold.
For every $p>6$ and $n$ such that 
$\tilde{\zeta}_n^2 < \tilde{\zeta}_n$,
it holds that 
\begin{eqnarray*}
P\left(\|\hat{\bm{\beta}}_n^0 - \bm{\beta}^0\|_2
\leq \delta_n\right)
\geq 1-\exp\left(
-b \log p
\right),
\end{eqnarray*}
where 
\[
\delta_n := \frac{\tilde{\gamma}_n \phi_{\max}+\sqrt{\tilde{\gamma}_n^2 \phi_{\max}^2+4 (\sigma-3\eta- \tilde{\gamma}_n)  \tilde{s} \lambda_0}}{2(\sigma-3 \eta - \tilde{\gamma}_n)}.
\]
\end{thm}

\begin{rem}
By the inequality of arithmetic and geometric means, it is easy to see that $\delta_n$ takes its lower bound of the order  when it holds that
\[
\phi_{\max}\tilde{\gamma}_n \asymp \sqrt{\phi_{\max}^2 \tilde{\gamma}_n^2 + 4(\sigma - 3\eta - \tilde{\gamma}_n) \tilde{s} \lambda_0}.
\]
In other words, $\tilde{s} \lambda_0 = O_p(\phi_{\max}^2\tilde{\gamma}_n^2)$.
This implies that $\delta_n = O_p(\phi_{\max} \tilde{\gamma}_n)$.
\end{rem}

Now we can establish the rate of convergence 
of the estimator $\hat{\bm{\beta}}_n^0$ as follows.
\begin{cor}\label{rate of convergence l0}
Suppose that $\sum_{l=0}^\infty \rho(l) < \infty$, $\log p \to \infty$ as $n \to \infty$, and that 
\[
\lambda_0 \asymp \phi_{\max}^2 \frac{\log p}{n}.
\]
Then, under the same assumptions in Theorem \ref{oracle l0}, it holds that 
\[
\|\hat{\bm{\beta}}_n^0 - \bm{\beta}^0\|_2^2
= O_p\left(
\tilde{s} \phi_{\max}^4 \frac{\log p}{n} 
\right),\quad
n \to \infty.
\]
\end{cor}

It can be seen from Corollaries \ref{rate of convergence Lasso} and \ref{rate of convergence l0} that
$l_0$-penalized estimators and Lasso-type estimators have similar rate of convergence $O_p(\log p/n)$ in terms of squared errors.
Looking into the factors in detail, we find that $\|\hat{\bm{\beta}}_n^1 - \bm{\beta}^0\|_2^2 \propto s_0^4 \phi_{\max}^2$
for Lasso-type estimators.
On the other hand,
for $l_0$-penalized estimators, we have $\|\hat{\bm{\beta}}_n^0 - \bm{\beta}^0\|_2^2 \propto s_0 \phi_{\max}^4$
under the situation $s \asymp s_0$.
Thus, the performance of Lasso-type estimators and $l_0$-penalized estimators has a tradeoff
between the largest eigenvalues and the sparsity of the true vector.

%

\section{Simulation studies.} \label{sec:6}
In this section, we investigate the finite sample performance of the sparse principal component analyses for stationary processes.
We also provide a real data example of average temperatures in Kyoto analyzed by the sparse PCA.

\subsection{Finite sample performance.} \label{subsec:6.1}
We generate the observation stretch from the model \eqref{eq:2.1}.
The first principal component vector $\bm{p}_1$ of the coefficient matrix $A$
is supposed to be 
\[
\bm{p}_1 = \frac{\bigl(f(1/p), \dots, f(p/p)\bigr)^{\top}}{\|\bigl(f(1/p), \dots, f(p/p)\bigr)\|},
\]
where the function $f$ is specified by each one of the functions in Figure \ref{fig:1}.
The function in the left figure is known as the three-peak function 
and
the function in the right figure is known as the step function.

\begin{figure}[H]
\begin{center}
 \subfigure 
 {\includegraphics[clip, width=0.45\columnwidth]{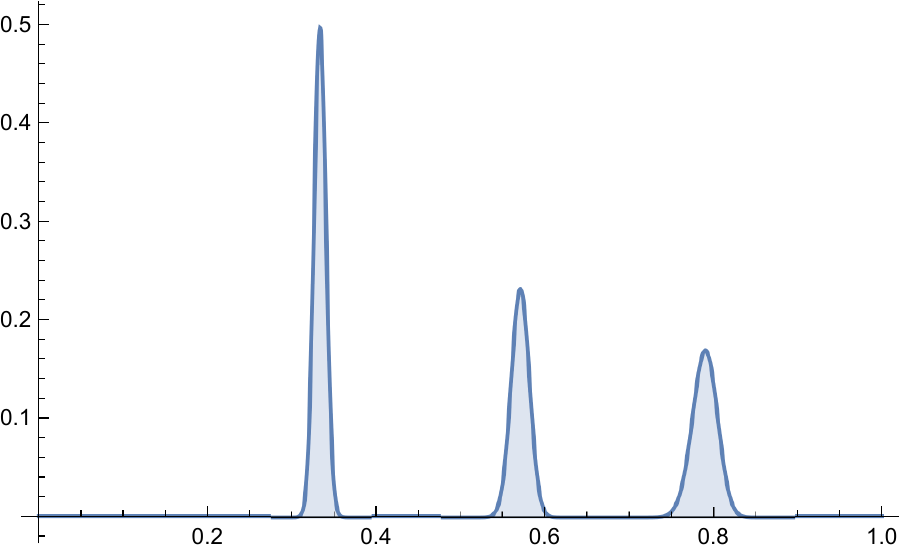}} \qquad 
 \subfigure 
 {\includegraphics[clip, width=0.45\columnwidth]{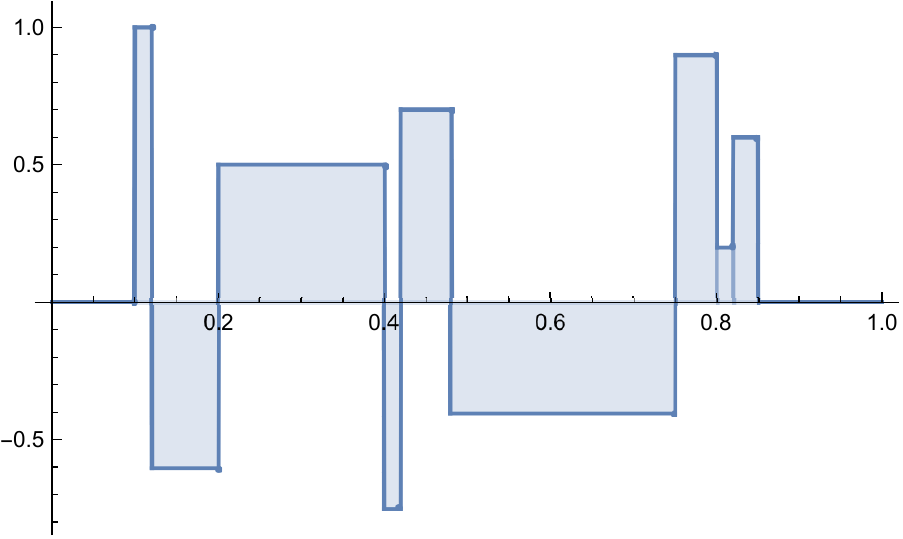}} 
\caption{Functions to generate the first principal component vector $\bm{p}_1$.}
\label{fig:1}
\end{center}
\end{figure}

The eigenvalues $\nu_1, \dots, \nu_p$ of $A$ are determined by $\nu_j = \nu^{j}$, $j = 1, \dots, p$,
where $\nu \in \{0.85, 0.6, 0.35, 0.1\}$.
Under this setting, the ratio of the first eigenvalue of the covariance matrix $\Sigma_0$ to the second eigenvalue 
is $1 + \nu^2$.
We compare the average squared errors $\|\hat{\bm{\beta}}_n - \bm{\beta}^0\|/p$
in $l_0$-penalized estimator, $l_1$-penalized estimator and standard principal component analysis.
The dimension $p$ of the stationary process is specified as $p = 512$, while
the number of the observation is specified as $n = 256$.
The penalty parameter $\lambda$ is taken as $\lambda_0 = 3 \log(p)/n$ and $\lambda_1 = (\log(p)/n)^{1/2}/3$.
The numerical results reported in Table \ref{tbl:1} are obtained over 1000 runs for 
the model \eqref{eq:2.1} with i.i.d.~innovations as centered Gaussian distribution 
and centered two-sided Weibull distribution with the shape parameter 0.5, respectively.
Here, the covariance matrix of innovations is the identity matrix.

From Table \ref{tbl:1}, we see that the penalized principal component analysis
performs better in terms of the loss than the standard principal component analysis for all cases.
The penalized principal component analyses show similar performance,
while the $l_1$-penalized estimator performs better in almost all cases in these simulations.
The $l_0$-penalized estimators are prone to choose a small number of features
while the $l_1$-penalized estimators balance the average squared error and the number of features.

\begin{table}[H]
\caption{Comparison of principal component analyses; Peak refers to the three-peak function
and Step refers to the step function.
Loss refers to the average squared error and Size refers to the number of the estimated elements in the first principal component.}
\begin{center}
\begin{tabular}{lccccc}
\hline
& & $l_0$-penalized estimator & $l_1$-penalized estimator & Standard PCA\\ \hline
{\bf Gaussian} \\ \hline
Vector & $\nu_1$ & Loss (Size)  & Loss (Size) & Loss \\[5pt] \hline
Peak & 0.85 & 0.00304 (29.01) & 0.00314 (47.92) & 0.00422 \\
 & 0.60 & 0.00350 (35.98) & 0.00297 (49.43) & 0.00513 \\
 & 0.35 & 0.00331 (35.73) & 0.00273 (49.42) & 0.00505 \\
 & 0.10 & 0.00324 (35.55) & 0.00265 (49.42) & 0.00501 \\
Step & 0.85 & 0.00410 (42.35) & 0.00411 (47.17) & 0.00424 \\
 & 0.60 & 0.00362 (47.03) & 0.00298 (51.58) & 0.00513 \\
 & 0.35 & 0.00342 (46.89) & 0.00272 (51.51) & 0.00505 \\
 & 0.10 & 0.00336 (46.98) & 0.00263 (51.41) & 0.00501 \\
 {\bf Weibull} \\ \hline
Vector & $\nu_1$ & Loss (Size)  & Loss (Size) & Loss \\[5pt] \hline
Peak & 0.85 & 0.00566 ( 5.33) & 0.00523 ( 9.07) & 0.00607 \\
 & 0.60 & 0.00523 (13.61) & 0.00470 (19.69) & 0.00595 \\
 & 0.35 & 0.00508 (13.68) & 0.00453 (19.37) & 0.00584 \\
 & 0.10 & 0.00500 (13.92) & 0.00444 (19.86) & 0.00578 \\
Step & 0.85 & 0.00587 (14.16) & 0.00540 (20.08) & 0.00603 \\
 & 0.60 & 0.00522 (14.01) & 0.00467 (20.04) & 0.00595 \\
 & 0.35 & 0.00506 (14.09) & 0.00449 (19.98) & 0.00582 \\
 & 0.10 & 0.00505 (13.91) & 0.00448 (19.73) & 0.00583 \\ \hline
 \end{tabular}
\end{center}
\label{tbl:1}
\end{table}%

\subsection{Real data example.}
We apply the sparse principal component analyses to the dataset of daily average temperatures in Kyoto, Japan.
The data are from January 1, 1901 to December 31, 2020, which are over the last 120 years.
We removed the temperature of February 29, if exists, to make each year have 365 days.
To summarize, the dimension of the data is 365 and the sample size is 120.

In general, the temperature data appear to increase over the years.
Assuming a linear trend in the temperature,
we removed the trend from the original data and obtained detrended data,
of which the partial data are shown in the left figure in Figure \ref{fig:2}.
The detrended temperature data appear to be stationary.
We also plot the eigenvalues of the covariance matrix obtained from the detrended data.
The spikiness condition in Assumption \ref{model setups} seems to be satisfied.
The sparsity feature of these data can be confirmed from the right figure in Figure \ref{fig:2}.

\begin{figure}[H]
\begin{center}
 \subfigure 
 {\includegraphics[clip, width=0.45\columnwidth]{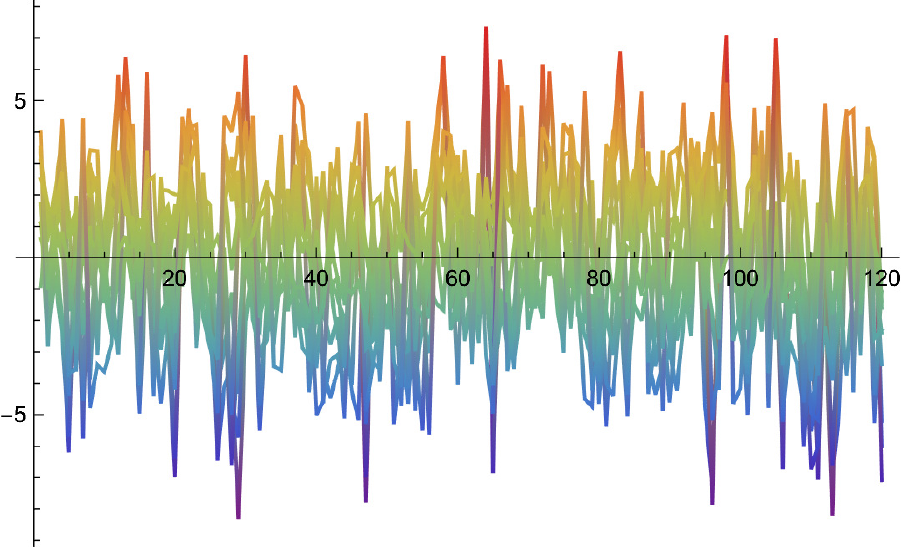}} \qquad 
 \subfigure 
 {\includegraphics[clip, width=0.45\columnwidth]{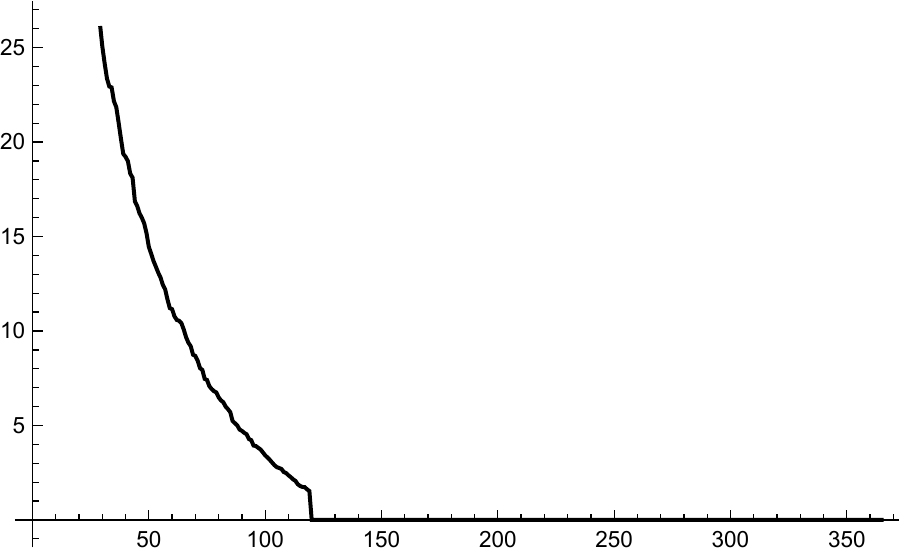}} 
\caption{(Left) The detrended temperature data from April 20 to April 30 over 120 years.
(Right) Eigenvalues of covariance matrix of the detrended temperature data.}
\label{fig:2}
\end{center}
\end{figure}

The sparse principal component analyses and the standard principal component analysis 
are applied to the data.
The penalty parameter $\lambda$ is taken as $\lambda_0 = 3 \log(p)/n$ and $\lambda_1 = (\log(p)/n)^{1/2}/3$
for $l_0$-penalty and $l_1$-penalty, respectively.
The numerical results are obtained as Figure \ref{fig:3}.
The nonzero coefficients are shown in colors with rainbow plots. The estimates of large absolute value are in red while those of small absolute value are in blue.
We can also find that the $l_0$-penalized estimate shows the smallest number of features in the first principal component, compared with other two methods,
while the $l_1$-penalized estimate harmonize the standard principal components analysis with the $l_0$-penalized one.

The sparse principal component analyses explain the feature of the temperatures in Kyoto well.
It is well known that  the temperature in Kyoto is radically going up and down during February and March over years.
Thus there is much more temperature variation during these months.
This result matches the data provided by Japan Meteorological Agency.
In summary, this feature is extracted by the sparse principal component analyses.

\begin{figure}[H]
\begin{center}
\includegraphics[clip, width=0.45\columnwidth]{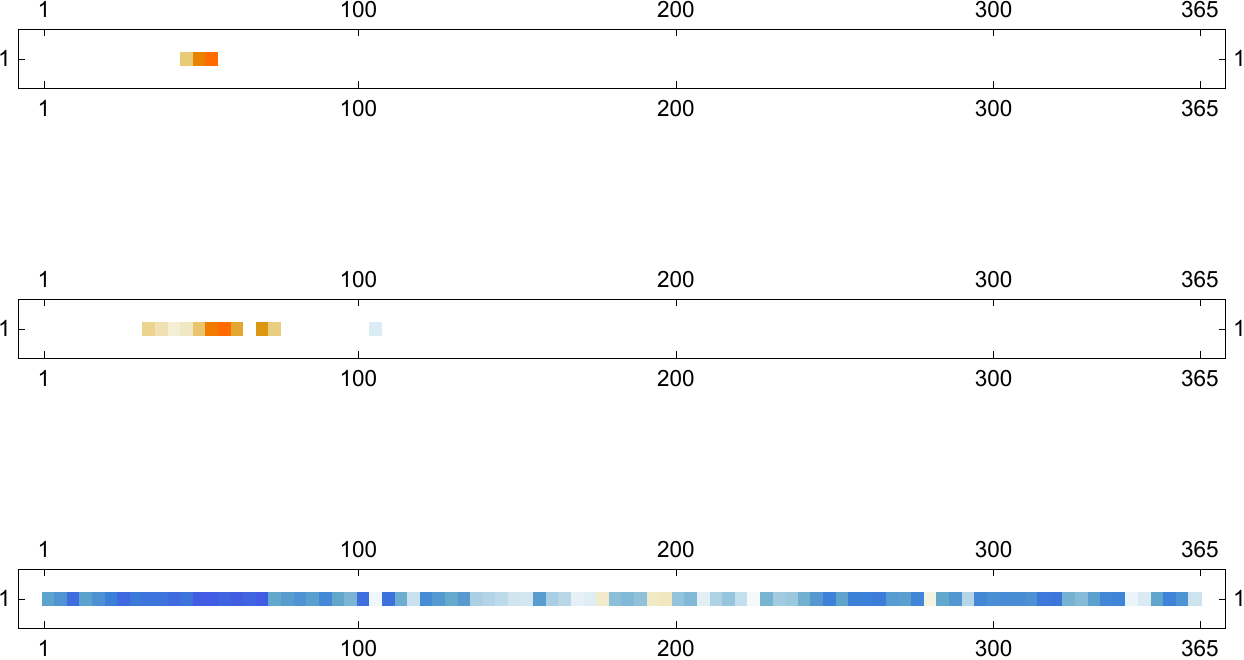} 
\caption{Plots for the estimated first principal components.
From the top to bottom, plots are for the $l_0$-penalized estimate, the $l_1$-penalized estimate,
and the standard one.}
\label{fig:3}
\end{center}
\end{figure}

\section{Proofs} \label{sec:7}
In this section, we complement the rigorous proofs for the main results in Sections \ref{sec:3}--\ref{sec:5}.
First, we summarize some crucial technical results in Subsections \ref{subsec:1} and \ref{subsec:2}
for $\alpha$-mixing Gaussian processes and $\beta$-mixing sub-Weibull processes, respectively.
The proofs for technical results in Subsection \ref{subsec:1} can be found in Subsection \ref{subsec:3},
while the proofs for technical results in Subsection \ref{subsec:2}
can be found in Subsection \ref{subsec:4}.
Subsection \ref{subsec:5} offers the proofs for $l_0$-penalized estimators.

\subsection{Technical results for Gaussian process.} \label{subsec:1}
We use the following concentration inequality, which is a modified form of the 
Hanson-Wright inequality. 
\begin{lem}\label{Hanson-Wright}
Let $\bm{Y} \sim N(\bm{0}, Q)$
be an $n$-dimensional normal random vector.
Then, there exists a universal constant 
$c>0$ such that for any $\eta>0$,
\[
P\left(
\frac{1}{n}\left|
\|\bm{Y}\|_2^2-E[\|\bm{Y}\|_2^2]
\right|
> \eta \|Q\|_2
\right)
\leq 2 \exp\left(
-cn \min \{\eta, \eta^2\}
\right).
\]
\end{lem}
See \cite{rudelson2013hanson}, \cite{basu2015regularized}, and 
\cite{wong2020lasso} for 
the detail of this inequality.

The next lemma guarantees that the empirical risk $R_n(\cdot)$ is also strictly convex on $\mathcal{B}$ with large probability.
\begin{lem}\label{empirical risk convex}
Suppose that the same assumptions as Proposition \ref{theoretical risk convex} and 
Assumption \ref{Gaussian} hold.
Then, for every $0 < \xi_n < \tau_n$, where 
\[
\tau_n := \frac{\sigma-3 \eta}{\sum_{l=0}^n \rho(l) \phi_{\max}},
\]
and for every unit vector $\bm{v} \in \mathbb{R}^p$, it holds that 
\[
P\left(|\bm{v}^\top W_n \bm{v}| > 
\xi_n \sum_{l=0}^n \rho(l) \phi_{\max}\right)
\leq 2 \exp(-cn \min(\xi_n,\xi_n^2)),
\]
which implies that
\[
P(\bm{v}^\top \ddot{R}_n(\bm{\beta})\bm{v} >0)
\geq 1-2 \exp(-cn \min(\xi_n,\xi_n^2)),\quad
\bm{\beta} \in \mathcal{B},
\]
with some universal constant $c>0$.
\end{lem}
For $\alpha$-mixing Gaussian processes, by Lemma \ref{empirical risk convex}, 
we find that for any convex penalty function ${\rm pen}(\cdot)$,
$R_n(\cdot) + \lambda \, {\rm pen}(\cdot)$ is still asymptotically strictly convex on $\mathcal{B}$,
which follows from the fact that the conical combination of convex functions is also convex.
This also implies that 
$\hat{\bm{\beta}}_n^1$ is a unique solution to the optimization problem $i.e.$, 
the Lasso-type PCA estimator is well-defined with large probability.

We establish the oracle inequality for the estimator $\hat{\bm{\beta}}_n^1$.
To do this, we should evaluate the
difference between the empirical risk and 
theoretical risk, which is achieved by the 
following lemma.
\begin{lem}\label{W ineq}
Let $b >0$ be a free parameter and $\tilde{c}>0$ be a constant.
Define that
\[
\zeta_n
:= \sqrt{\frac{(b+1)\log p^2}{\tilde{c} n}}.
\]
Under the Assumptions \ref{model setups}, \ref{Gaussian}
and the same assumption as Proposition \ref{theoretical risk convex}, for 
every $n$ satisfying  
$\log p/n \leq 1$, it holds that
\[
P\left(
\|W_n\|_{\max} \leq \zeta_n \sum_{l=0}^n \rho(l) \phi_{\max}
\right)
\geq 1-\exp(-b \log p^2),
\]
where $W_n= \hat{\Sigma}_n - \Sigma_0$.
\end{lem}

\subsection{Technical results for sub-Weibull process.} \label{subsec:2}
To derive the oracle inequality for the $\beta$-mixing sub-Weibull process,
we use the following concentration 
inequality, which is established by 
\cite{merlevede2011bernstein}.

\begin{lem}\label{concentration sub-Weibull}
Let $\{X_t\}_{t \in \mathbb{Z}}$ be an 
$\mathbb{R}$-valued zero mean strictly stationary process,
which satisfies that 
\[
\beta(n) \leq 2 \exp(-c n^{\gamma_1}),
\]
and
\[
\|X_t\|_{\psi_{\gamma_2}} \leq K,\quad
\forall t \in \mathbb{Z},
\]
for constants $c, \gamma_1, \gamma_2,\ K>0$.
Let $\gamma$ be 
\[
\gamma := \left(
\frac{1}{\gamma_1} + \frac{1}{\gamma_2}
\right)^{-1} < 1.
\]
Then, for every $n >4$ and $\epsilon > 1/n$,
it holds that 
\begin{eqnarray*}
P\left(
\left|
\frac{1}{n} \sum_{t=1}^n X_t
\right| > \epsilon
\right)
&\leq& 
n \exp\left(
-\frac{(\epsilon n)^\gamma}{K^\gamma C_1}
\right)
+ \exp\left(
- \frac{\epsilon^2 n}{K^2 C_2}
\right),
\end{eqnarray*}
where $C_1$ and $C_2$ are constants depending only on $\gamma_1, \gamma_2$
and $c$.
\end{lem}
Using Lemma \ref{concentration sub-Weibull}, we obtain the following result 
which is corresponding to the 
Lemma \ref{empirical risk convex} for Gaussian case.
\begin{lem}\label{empirical risk convex sub-Weibull}
Suppose that the same assumptions as Proposition \ref{theoretical risk convex} and 
Assumption \ref{sub-Weibull process} hold.
Then, for every $\xi_n$ such that 
\[
\frac{K_2 C_1^{1/\gamma}(\log n)^{1/\gamma}}{n}
< \xi_n < \sigma-3\eta,
\]
where 
\[
K_2 := 2^{2/\gamma_2} K^2
\]
and for every unit vector $\bm{v} \in \mathbb{R}^p$, it holds that 
\[
P\left(|\bm{v}^\top W_n \bm{v}| > 
\xi_n \right)
\leq 2n \exp\left(-\frac{(\xi_n n)^\gamma}{K_2^\gamma C_1}\right)
+ 2 \exp\left(
-\frac{\xi_n^2 n}{K_2^2 C_1}
\right),
\]
where $C_1$ and $C_2$ are constants
depending only on $c,\ \gamma_1,\  \gamma_2$ described in Assumption \ref{sub-Weibull process}.
Especially, it holds that 
\[
P(\bm{v}^\top \ddot{R}_n(\bm{\beta})\bm{v} >0)
\geq 1-2n \exp\left(-\frac{(\xi_n n)^\gamma}{K_2^\gamma C_1}\right)
- 2 \exp\left(
-\frac{\xi_n^2 n}{K_2^2 C_2}
\right),\quad
\bm{\beta} \in \mathcal{B}.
\]
\end{lem}
Note that if we take $\xi_n$ as 
\[
\xi_n = \frac{K_2 C_2^{1/\gamma}(\log n)^{1/\gamma}}{n^{1-\alpha}},\quad
\alpha \in (0, 1/2),
\]
for sufficiently large $n$,
then, it holds that 
\[
P\left(|\bm{v}^\top W_n \bm{v}| > 
\xi_n \right) \to 0,\quad \text{as $n \to \infty$},
\]
which implies that 
\[
P(\bm{v}^\top \ddot{R}_n(\bm{\beta})\bm{v} >0)
\to 1,\quad
\bm{\beta} \in \mathcal{B}.
\]
Therefore, we conclude that for $\beta$-mixing sub-Weibull processes, 
$R_n(\cdot) + \lambda \, {\rm pen}(\cdot)$
is also asymptotically strictly convex on $\mathcal{B}$ with large probability for any convex penalty function ${\rm pen}(\cdot)$.
As for the bound for $\|W_n\|_{\max}$, we have the following lemma.
\begin{lem}\label{W ineq sub-Weibull}
Suppose that Assumptions \ref{model setups} and
\ref{sub-Weibull process} hold.
For every $\zeta_n$ satisfying
\[
\zeta_n > \max\left\{
\frac{K_2 C_1^{1/\gamma}(\log n p^2 + 2 \tilde{b} \log p)^{1/\gamma}}{n},
K_2 \sqrt{\frac{2 C_2 (\tilde{b} + 1)\log p}{n}}
\right\},
\]
where $\tilde{b} > 0$ is a free parameter,
it holds that
\[
P\left(
\|W_n\|_{\max} \leq \zeta_n
\right)
\geq 1-2 \exp(- \tilde{b} \log p^2).
\]
\end{lem}

\subsection{Proofs for Subsection \ref{subsec:1} and Section \ref{sec:3}.} \label{subsec:3}
In this subsection, we provide proofs for main results and technical results for Lasso-type estimator for $\alpha$-mixing Gaussian process.
\begin{proof}[Proof of Lemma \ref{empirical risk convex}]
Noting that 
\[
\ddot{R}_n(\bm{\beta})
= \ddot{R}(\bm{\beta})
-(\hat{\Sigma}_n - \Sigma_0),
\]
we have that for every unit vector $\bm{v} \in \mathbb{R}^p$,
\begin{eqnarray*}
\bm{v}^\top \ddot{R}_n(\bm{\beta}) \bm{v}
&=& \bm{v}^\top \ddot{R}_n(\bm{\beta})\bm{v}
- \bm{v}^\top W_n \bm{v}\\
&\geq& 2(\sigma-3 \eta) - \bm{v}^\top W_n \bm{v},
\end{eqnarray*}
where $W_n:= \hat{\Sigma}_n - \Sigma_0$.
Therefore, it suffices to evaluate the probability that 
$P(\bm{v}^\top W_n \bm{v} > \eta_n)$
for some $\eta_n < \sigma-3 \eta$.
Put $\bm{X}_{(n)}=(\bm{X}_1,\ldots,\bm{X}_n)$.
It follows from the stationarity of $\{\bm{X}_t\}_{t\in \mathbb{Z}}$ that 
\[
\hat{\Sigma}_n = \frac{1}{n} \bm{X}_{(n)} \bm{X}_{(n)}^\top,\quad
\Sigma_0 = \frac{1}{n} E[\bm{X}_{(n)} \bm{X}_{(n)}^\top],
\]
which implies that 
\begin{eqnarray*}
\bm{v}^\top W_n \bm{v}
&=& \frac{1}{n}\bm{v}^\top (\bm{X}_{(n)} \bm{X}_{(n)}^\top - E[\bm{X}_{(n)} \bm{X}_{(n)}^\top])\bm{v} \\
&=&\frac{1}{n}(\|\bm{X}_{(n)}^\top \bm{v}\|_2^2-E[\|\bm{X}_{(n)}^\top \bm{v}\|_2^2]).
\end{eqnarray*}
Let $Q_n$ be the covariance matrix of the random variable $\bm{X}_{(n)}^\top \bm{v}$.
By the simple calculation, we can find that
\[
Q_n = \left(\begin{array}{ccc}
\bm{v}^\top E[\bm{X}_1\bm{X}_1^\top]\bm{v} & \cdots & \bm{v}^\top E[\bm{X}_1\bm{X}_n^\top]\bm{v} \\
\vdots & \ddots & \vdots \\
\bm{v}^\top E[\bm{X}_n\bm{X}_1^\top]\bm{v} & \cdots & \bm{v}^\top E[\bm{X}_n\bm{X}_n^\top]\bm{v}
\end{array}
\right).
\]
Noting that $\{\bm{X}_t\}_{t \in \mathbb{Z}}$
is a centered Gaussian time series, 
we have that 
$\bm{X}_{(n)}^\top \bm{v} \sim N(\bm{0}, Q_n)$.
We can apply Lemma \ref{Hanson-Wright}
to $\bm{X}_{(n)}^\top \bm{v}$ to deduce that
for every $\xi>0$,
\begin{eqnarray*}
P\left(
\frac{1}{n}\left|
\|\bm{X}_{(n)}^\top \bm{v}\|_2^2
- E[\|\bm{X}_{(n)}^\top \bm{v}\|_2^2]
\right| > \xi \|Q_n\|_2
\right)
\leq 2 \exp(-cn \min\{\xi, \xi^2\}),
\end{eqnarray*}
where $c>0$ is a universal constant.
Noting that the $\alpha$-mixing 
Gaussian time series $\{\bm{X}_t\}_{t \in \mathbb{Z}}$ is also $\rho$-mixing,
we have
\[
\|Q_n\|_2 \leq \sum_{l=0}^n \rho(l) \|\Sigma_0\|_2 = \sum_{l=0}^n \rho(l) \phi_{\max}.
\]
We therefore obtain that 
\[
P\left(
\frac{1}{n}\left|
\|\bm{X}_{(n)}^\top \bm{v}\|_2^2
- E[\|\bm{X}_{(n)}^\top \bm{v}\|_2^2]
\right| > \xi \sum_{l=0}^n \rho(l) \phi_{\max}
\right)
\leq 2 \exp(-cn \min\{\xi, \xi^2\}).
\]
Especially, we can take a sequence $\xi_n$ satisfying that
\[
0 < \xi_n < \frac{\sigma-3\eta}{\sum_{l=0}^n \rho(l) \phi_{\max}},
\]
which concludes the lemma.
\end{proof}
\begin{proof}[Proof of Lemma \ref{W ineq}]
For every $i, j \in \{1,\ldots,p\}$, it holds that
\[
W_{nij} = \bm{e}_i^\top W_n \bm{e}_j,
\]
where $\bm{e}_k,\ k=1,\ldots,p$ is the $k$-th canonical basis of $\mathbb{R}^p$. 
Since 
\[
W_n = \hat{\Sigma}_n-\Sigma_0
= \frac{1}{n} (\bm{X}_{(n)} \bm{X}_{(n)}^\top - E[\bm{X}_{(n)} \bm{X}_{(n)}^\top]),
\]
it holds that 
\[
W_{nij}
= \frac{1}{n}\left(\bm{e}_i^\top \bm{X}_{(n)} \bm{X}_{(n)}^\top \bm{e}_j - E[\bm{e}_i^\top \bm{X}_{(n)} \bm{X}_{(n)}^\top \bm{e}_j]\right).
\]
Putting $\bm{Y}_k = \bm{X}_{(n)}^\top \bm{e}_k,\ k=1,\ldots,p$, we can rewrite that
\[
W_{nij} = \frac{1}{n} \left(
\bm{Y}_i^\top \bm{Y}_j
- E[\bm{Y}_i^\top \bm{Y}_j]
\right).
\]
Note that 
\begin{eqnarray*}
\bm{Y}_i^\top \bm{Y}_j
- E[\bm{Y}_i^\top \bm{Y}_j]
&=& \frac{1}{2} \left\{(\|\bm{Y}_i + \bm{Y}_j\|_2^2- E[\|\bm{Y}_i + \bm{Y}_j\|_2^2])\right.\\
&& \left.
-(\|\bm{Y}_i\|_2^2-E[\|\bm{Y}_i\|_2^2])
-(\|\bm{Y}_j\|_2^2-E[\|\bm{Y}_j\|_2^2])
\right\}.
\end{eqnarray*}
We then find that 
\begin{eqnarray*}
|W_{nij}|
&=& \frac{1}{n} \left|
\bm{Y}_i^\top \bm{Y}_j
- E[\bm{Y}_i^\top \bm{Y}_j]
\right| \\
&\leq& \frac{1}{2n} \left|
(\|\bm{Y}_i + \bm{Y}_j\|_2^2- E[\|\bm{Y}_i + \bm{Y}_j\|_2^2]
\right| \\
&& + \frac{1}{2n} \left|
\|\bm{Y}_i\|_2^2-E[\|\bm{Y}_i\|_2^2]
\right|
+ \frac{1}{2n} \left|
\|\bm{Y}_j\|_2^2-E[\|\bm{Y}_j\|_2^2]
\right|.
\end{eqnarray*}
Also, $\bm{Y}_i,\ \bm{Y}_j,\ \bm{Y}_i + \bm{Y}_j$ are centered Gaussian random 
variables.
Denote the covariance matrices of them by
$\Sigma_{Y_i},\ \Sigma_{Y_j},$ and $\Sigma_{Y_i+Y_j}$, respectively. 
Then, the following inequalities directly follow from Lemma \ref{Hanson-Wright} that
for every $\zeta >0$, there exists a universal constant $\tilde{c}>0$ such that,
\begin{equation}\label{W 1}
P\left(
\frac{1}{n} \left|
\|\bm{Y}_i\|_2^2 - E[\|\bm{Y}_i\|_2^2]
\right| > \zeta \|\Sigma_{Y_i}\|_2
\right)
\leq 2 \exp(-\tilde{c} n \min\{\zeta,\zeta^2\}),
\end{equation}
\begin{equation}\label{W 2}
P\left(
\frac{1}{n} \left|
\|\bm{Y}_j\|_2^2 - E[\|\bm{Y}_j\|_2^2]
\right| > \zeta \|\Sigma_{Y_j}\|_2
\right)
\leq 2 \exp(-\tilde{c} n \min\{\zeta,\zeta^2\}),
\end{equation}
and
\[
P\left(
\frac{1}{n} \left|
\|\bm{Y}_i + \bm{Y}_j\|_2^2 - E[\|\bm{Y}_i + \bm{Y}_j\|_2^2]
\right| > \zeta \|\Sigma_{Y_i+Y_j}\|_2
\right)
\leq 2 \exp(-\tilde{c} n \min\{\zeta,\zeta^2\}).
\]
After some tedious computation, we have
\[
\|\Sigma_{Y_i + Y_j}\|_2
\leq \frac{\|\Sigma_{Y_i}\|_2+\|\Sigma_{Y_j}\|_2}{2}.
\]
Therefore, we have that 
\begin{eqnarray}\label{W 3}
P\left(
\frac{1}{n} \left|
\|\bm{Y}_i + \bm{Y}_j\|_2^2 - E[\|\bm{Y}_i + \bm{Y}_j\|_2^2]
\right| \right.&>&\left. \frac{\zeta(\|\Sigma_{Y_i}\|_2+\|\Sigma_{Y_j})\|_2}{2}
\right) \nonumber \\
&\leq& 2 \exp(-\tilde{c} n \min\{\zeta,\zeta^2\}).
\end{eqnarray}
The inequalities (\ref{W 1})--(\ref{W 3}) imply that for every $\zeta >0$, there exists a constant $\tilde{c}$ such that 
\[
P(|W_{nij}| > \zeta(\|\Sigma_{Y_i}\|_2 + \|\Sigma_{Y_i}\|_2))
\leq 6 \exp(-\tilde{c} n \min\{\zeta, \zeta^2\}).
\]
Let $b >0$ be a free parameter.
We take $\zeta$ as 
\[
\zeta := \zeta_n = \sqrt{\frac{(b+1)\log p^2}{\tilde{c} n}}.
\]
Then, for every $n$ such that $\zeta_n^2 < \zeta_n$, it holds that
\begin{eqnarray*}
P\left(\|W_n\|_{\max} > 4\zeta\sum_{l=0}^n \rho(l) \phi_{\max} 
\right)
&\leq&
P\left(\|W_n\|_{\max} > \max_{i,j}\zeta(\|\Sigma_{Y_i}\|_2 + \|\Sigma_{Y_j}\|_2)\right) \\
&=& P\left(
\max_{i, j} |\bm{e}_i^\top W_n \bm{e}_j|
>\max_{i,j}\zeta(\|\Sigma_{Y_i}\|_2 + \|\Sigma_{Y_j}\|_2)
\right) \\
&\leq&
\sum_{i, j} P\left( |\bm{e}_i^\top W_n \bm{e}_j|>
\zeta(\|\Sigma_{Y_i}\|_2 + \|\Sigma_{Y_j}\|_2)
\right)\\
&\leq& 6p^2 \exp(-\tilde{c}n \min\{\zeta,\zeta^2\}) \\
&=& 6 \exp\left(\log p^2 - \tilde{c} n
\frac{(b+1)\log p^2}{\tilde{c} n} \right) \\
&=& 6 \exp(-b \log p^2),
\end{eqnarray*}
which completes the proof.
\end{proof}
\begin{proof}[Proof of Theorem \ref{oracle ineq Lasso}]
Let $\eta_n$ and $\gamma_n$ be
\[
\eta_n = \xi_n \sum_{l=0}^n \rho(l) \phi_{\max},
\]
and
\[
\gamma_n = \zeta_n \sum_{l=0}^n \rho(l)\phi_{\max}, 
\]
respectively.
Under the constraints $\xi_n < \tau_n : = \frac{\sigma-3\eta}{\sum_{l=0}^n \rho(l) \phi_{\max}}$, it suffices to show the inequality
\begin{equation} \label{eq:7.4l}
\|\hat{\bm{\beta}}_n^1-\bm{\beta}^0\|_1
\leq \frac{2(C+1)^2 \lambda_1 s_0}{\sigma-3\eta-\eta_n}
\end{equation}
on the event
\[
\{\|W_n\|_{\max} \leq \gamma_n\} \cap
\{|\bm{v}^\top W_n \bm{v}| \leq \eta_n\}
\]
for 
\[
\bm{v} = \frac{\hat{\bm{\beta}}^1_n - \bm{\beta}^0}{\|\hat{\bm{\beta}}^1_n - \bm{\beta}^0\|_2}.
\]
Following Lemma 7.1 of \cite{van2016estimation}, it holds that
\begin{equation}\label{two point ineq}
-\dot{R}_n(\hat{\bm{\beta}}_n^1)^\top(\bm{\beta}^0-\hat{\bm{\beta}}_n^1)
\leq \lambda_1 \|\bm{\beta}^0\|_1 - 
\lambda_1 \|\hat{\bm{\beta}}_n^1\|_1.
\end{equation}
Moreover, it follows from Proposition \ref{theoretical risk convex} and the Taylor 
expansion that 
\begin{equation}\label{two point margin}
R(\bm{\beta}^0)-R(\hat{\bm{\beta}}_n^1)
- \dot{R}(\hat{\bm{\beta}}_n^1)^\top(\bm{\beta}^0-\hat{\bm{\beta}}_n^1)
\geq (\sigma-3\eta)\|\bm{\beta}^0-\hat{\bm{\beta}}^1_n\|_2^2 \geq 0.
\end{equation}
Combining (\ref{two point ineq}) and (\ref{two point margin}), we have that 
\begin{equation} \label{eq:7.7ll}
R(\hat{\bm{\beta}}^1_n)-R(\bm{\beta}^0) + \lambda_1 \|\hat{\bm{\beta}}^1_n\|_1
\leq 
\bigl(\dot{R}_n(\hat{\bm{\beta}}^1_n)-\dot{R}(\hat{\bm{\beta}}^1_n) \bigr)^\top (\bm{\beta}^0 - \hat{\bm{\beta}}^1_n) + \lambda_1 \|\bm{\beta}^0\|_1.
\end{equation}
Noting that
\[
\dot{R}_n(\hat{\bm{\beta}}^1_n) - \dot{R}(\hat{\bm{\beta}}^1_n)
= - W_n \hat{\bm{\beta}}^1_n,
\]
we have, on the event $\{|\bm{v}^\top W_n \bm{v}| \leq \eta_n\}$,
\begin{align}
\bigl(\dot{R}_n(\hat{\bm{\beta}}^1_n)-\dot{R}(\hat{\bm{\beta}}^1_n)\bigr)^\top (\bm{\beta}^0 - \hat{\bm{\beta}}^1_n) 
&= -\hat{\bm{\beta}}_n^{1\top} W_n (\bm{\beta}^0 - \hat{\bm{\beta}}^1_n)  \notag \\
&= (\hat{\bm{\beta}}^1_n - \bm{\beta}^0)^\top W_n (\hat{\bm{\beta}}^1_n - \bm{\beta}^0) + \bm{\beta}^{0\top}W_n(\hat{\bm{\beta}}^1_n - \bm{\beta}^0)\notag \\
&\leq \eta_n \|\hat{\bm{\beta}}^1_n - \bm{\beta}^0\|_2^2+ \|W_n\|_{\max} \|\bm{\beta}^0\|_1 \|\hat{\bm{\beta}}^1_n-\bm{\beta}^0\|_1. \label{eq:7.9ll}
\end{align}
Since $\bm{\beta}^0$ satisfies that 
$\dot{R}(\bm{\beta}^0)=\bm{0}$,
it follows from Proposition \ref{theoretical risk convex} and the Taylor expansion that
\[
R(\hat{\bm{\beta}}^1_n) - R(\bm{\beta}^0)
\geq (\sigma-3 \eta)\|\hat{\bm{\beta}}^1_n-\bm{\beta}^0\|_2^2,
\]
which implies 
\begin{equation} \label{eq:7.10ll}
R(\hat{\bm{\beta}}^1_n) - R(\bm{\beta}^0) + \lambda_1 \| \hat{\bm{\beta}}_n^1\|
\geq (\sigma-3 \eta)\|\hat{\bm{\beta}}^1_n-\bm{\beta}^0\|_2^2 + \lambda_1 \| \hat{\bm{\beta}}_n^1\|.
\end{equation}
Therefore, by \eqref{eq:7.10ll}, \eqref{eq:7.7ll} and \eqref{eq:7.9ll}, we find that, on the event $\{\|W_n\|_{\max} \leq \gamma_n\}$,
\[
(\sigma-3 \eta) \|\hat{\bm{\beta}}^1_n - \bm{\beta}^0\|_2^2 + \lambda_1 \|\hat{\bm{\beta}}^1_n\|_1
\leq \eta_n \|\hat{\bm{\beta}}^1_n - \bm{\beta}^0\|_2^2 + \gamma_n \|\bm{\beta}^0\|_1 \|\hat{\bm{\beta}}^1_n-\bm{\beta}^0\|_1 + \lambda_1 \|\bm{\beta}^0\|_1,
\]
and thus,
\begin{equation} \label{eq:7.6l}
(\sigma-3 \eta - \eta_n) \|\hat{\bm{\beta}}^1_n -
\bm{\beta}^0\|_2^2
\leq \gamma_n \|\bm{\beta}^0\|_1 \|\hat{\bm{\beta}}^1_n-\bm{\beta}^0\|_1 
+ \lambda_1 \|\bm{\beta}^0\|_1 - \lambda_1\|\hat{\bm{\beta}}^1_n\|_1.
\end{equation}
The right-hand side of \eqref{eq:7.6l} can be bounded as follows.
\begin{align}
& \gamma_n \|\bm{\beta}^0\|_1 \|\hat{\bm{\beta}}^1_n-\bm{\beta}^0\|_1 
+ \lambda_1 \|\bm{\beta}^0\|_1 - \lambda_1\|\hat{\bm{\beta}}^1_n\|_1 \notag\\
= \quad & \gamma_n \|\bm{\beta}^0\|_1 
\|\hat{\bm{\beta}}^1_{n S} - \bm{\beta}^0_{S}\|_1 + \gamma_n \|\bm{\beta}^0\|_1 
\|\hat{\bm{\beta}}^1_{n S^c} - \bm{\beta}^0_{S^c}\|_1 \notag\\
&  + \lambda_1 \|\bm{\beta}^0_S\|_1
+ \lambda_1 \|\bm{\beta}^0_{S^c}\|_1
-\lambda_1 \|\hat{\bm{\beta}}^1_{nS}\|_1
-\lambda_1 \|\hat{\bm{\beta}}^1_{nS^c}\|_1 \notag\\
\leq \quad & (\lambda_1+ \gamma_n \|\bm{\beta}^0\|_1) \|\hat{\bm{\beta}}^1_{nS} - \bm{\beta}^0_{S}\|_1
- (\lambda_1- \gamma_n \|\bm{\beta}^0\|_1) \|\hat{\bm{\beta}}^1_{nS^c} - \bm{\beta}^0_{S^c}\|_1. \label{eq:7.8l}
\end{align}
Since $\eta_n < \sigma-3\eta$,
we see that the left-hand side of \eqref{eq:7.6l} is positive. Hence, we have that 
\begin{equation} \label{eq:7.9l}
0 <  (\lambda_1+ \gamma_n \|\bm{\beta}^0\|_1) \|\hat{\bm{\beta}}^1_{nS} - \bm{\beta}^0_{S}\|_1
- (\lambda_1- \gamma_n \|\bm{\beta}^0\|_1) \|\hat{\bm{\beta}}^1_{nS^c} - \bm{\beta}^0_{S^c}\|_1.
\end{equation}
Now let $\Delta := \hat{\bm{\beta}}^1_n-\bm{\beta}^0$.
By \eqref{eq:7.9l}, we then find that
\[
\|\Delta_{S^c}\|_1
\leq \frac{\lambda_1+ \gamma_n \|\bm{\beta}^0\|_1}{\lambda_1- \gamma_n \|\bm{\beta}^0\|_1} \|\Delta_S\|_1
\leq C \|\Delta_S\|_1.
\]
Consequently, we have that 
\[
\|\Delta\|_1 \leq (C+1)\|\Delta_S\|_1,
\]
In view of $\|\Delta_S\|_1 \leq \sqrt{s_0}\|\Delta\|_2$, we then obtain
\[
\frac{\|\Delta\|_1}{(C+1)\sqrt{s_0}} \leq \frac{\|\Delta_S\|_1}{\sqrt{s_0}}  \leq \|\Delta\|_2,
\]
which implies that 
\[
(\sigma-3\eta - \eta_n)\|\Delta\|_2^2 
\geq 
\frac{\sigma-3\eta-\eta_n}{(C+1)^2s_0} \|\Delta\|_1^2.
\]
Using \eqref{eq:7.6l} and \eqref{eq:7.8l}, we finally obtain that 
\begin{eqnarray*}
\frac{\sigma-3\eta-\eta_n}{(C+1)^2s_0}\|\Delta\|_1^2
&\leq& (\lambda_1 + \gamma_n \|\bm{\beta}^0\|_1) \|\Delta_S\|_1 - (\lambda_1 - \gamma_n \|\bm{\beta}^0\|_1) \|\Delta_{S^c}\|_1 \\
&\leq& (\lambda_1 + \gamma_n \|\bm{\beta}^0\|_1) \|\Delta_S\|_1 + (\lambda_1 - \gamma_n \|\bm{\beta}^0\|_1) \|\Delta_{S^c}\|_1 \\
&\leq& (\lambda_1 + \gamma_n \|\bm{\beta}^0\|_1) \|\Delta\|_1 + (\lambda_1 - \gamma_n \|\bm{\beta}^0\|_1) \|\Delta\|_1\\
&=& 2 \lambda_1 \|\Delta\|_1,
\end{eqnarray*}
which ends the proof of \eqref{eq:7.4l}.
\end{proof}
\subsection{Proofs for Subsection \ref{subsec:2} and Section \ref{sec:4}.} \label{subsec:4}
In this subsection, we provide proofs for main results and technical results for Lasso-type estimator for $\beta$-mixing sub-Weibull process.
\begin{proof}[Proof of Lemma \ref{empirical risk convex sub-Weibull}]
For every $\bm{v} \in \mathbb{S}^{p-1}$, it holds that
\[
\bm{v}^\top W_n \bm{v}
= \frac{1}{n}(\|\bm{X}_{(n)}^\top \bm{v}\|_2^2-E[\|\bm{X}_{(n)}^\top \bm{v}\|_2^2]).
\]
Define the process $\{Z_t(\bm{v})\}_{t\in \mathbb{Z}}$ by 
\[
Z_t(\bm{v})
= |\bm{v}^\top \bm{X}_t|^2 - E[|\bm{v}^\top \bm{X}_t|^2].
\]
Then, we have that
\[
\|\bm{X}_{(n)}^\top \bm{v}\|_2^2-
E[\|\bm{X}_{(n)}^\top \bm{v}\|_2^2]
= \sum_{t=1}^n Z_t(\bm{v}).
\]
Since $\{\bm{X}_t\}_{t \in \mathbb{Z}}$ is the sub-Weibull $(\gamma_2)$,
the process $\{Z_t(\bm{v})\}_{t \in \mathbb{Z}}$ is 
the sub-Weibull $(\gamma_2/2)$.
Therefore, it follows from Lemma \ref{concentration sub-Weibull} that 
\begin{eqnarray*}
P(|\bm{v}^\top W_n \bm{v}|> \xi_n)
&=&
P\left(
\left|
\frac{1}{n} \sum_{t=1}^n Z_t(\bm{v})
\right| > \xi_n
\right) \\
&\leq& 2n \exp\left(
- \frac{(\xi_n n)^\gamma}{K_2 C_1}
\right)
+ \exp\left(
- \frac{\xi_n^2 n}{K_2^2 C_2}
\right),
\end{eqnarray*}
which completes the proof.
\end{proof}
\begin{proof}[Proof of Lemma \ref{W ineq sub-Weibull}]
As well as Proof of Lemma \ref{W ineq}, 
we can find that
\begin{eqnarray*}
|W_{nij}|
&\leq& \frac{1}{2n} \left|
(\|\bm{Y}_i + \bm{Y}_j\|_2^2- E[\|\bm{Y}_i + \bm{Y}_j\|_2^2]
\right| \\
&& + \frac{1}{2n} \left|
\|\bm{Y}_i\|_2^2-E[\|\bm{Y}_i\|_2^2]
\right|
+ \frac{1}{2n} \left|
\|\bm{Y}_j\|_2^2-E[\|\bm{Y}_j\|_2^2]
\right|,
\end{eqnarray*}
where $\bm{Y}_k = \bm{X}_{(n)}^\top \bm{e}_k,\ k=1,\ldots,p$.
Note that 
\[
\|\bm{Y}_i\|_2^2
= \sum_{t=1}^n X_{i,t}^2,\quad
\|\bm{Y}_i + \bm{Y}_j\|_2^2
= \sum_{t=1}^n (X_{i,t}^2+X_{j,t}^2),\quad
i, j = 1,\ldots,p,
\]
where $X_{i,t}$ is the $i$-th component of $\bm{X}_t$.
Since the process $\{\bm{X}_t\}_{t \in \mathbb{Z}}$ is sub-Weibull $(\gamma_2)$, 
it follows 
that $\{X_{i,t}^2\}_{t \in \mathbb{Z}}$ is sub-Weibull $(\gamma_2/2)$.
Therefore, we can apply Lemma \ref{empirical risk convex sub-Weibull} to deduce the
conclusion.
\end{proof}
\begin{proof}[Proof of Theorem \ref{oracle ineq sub-Weibull}]
It suffices to show the inequality on the event
\[
\{\|W_n\|_{\max} \leq \zeta_n\}\cap\{\bm{v}^\top W_n \bm{v} \leq \xi_n\},
\]
where
\[
\bm{v} = \frac{\hat{\bm{\beta}}_n^1-\bm{\beta}^0}{\|\hat{\bm{\beta}}_n^1-\bm{\beta}^0\|_2}.
\]
The remaining part of the proof is similar to the proof of
Theorem \ref{oracle ineq Lasso}.
\end{proof}
\subsection{$l_0$-penalized estimator for $\alpha$-mixing Gaussian process.} \label{subsec:5}
In this subsection, we prove main results for the $l_0$-penalized estimator for $\alpha$-mixing process.
\begin{lem}\label{maximal ineq l0}
Suppose that Assumptions \ref{model setups} and \ref{Gaussian} hold.
Let $\zeta_n$ be
\[
\tilde{\zeta}_n = \sqrt{\frac{(b+4 \tilde{s}) \log p}{cn}},
\]
where $b>0$ is a free parameter and $c>0$ 
is a constant.
For every $p>6$ and $n$ such that 
$\tilde{\zeta}_n^2 < \tilde{\zeta}_n$, it holds that 
\[
P\left(
\sup_{v \in \mathcal{B}(2 \tilde{s})\cap \mathbb{S}^{p-1}} |\bm{v}^\top W_n \bm{v}| >
\tilde{\zeta}_n \sum_{l=0}^n \rho(l) \phi_{\max}
\right)
\leq 
\exp\left(
-b \log p
\right).
\] 
\end{lem}

\begin{proof}[Proof of Lemma \ref{maximal ineq l0}]
First, we fix $\bm{v} \in \mathcal{B}_0(2 \tilde{s})\cap \mathbb{S}^{p-1}$ arbitrarily.
In view of the proof of Lemma \ref{empirical risk convex}, we have that 
\[
|\bm{v}^\top W_n \bm{v}|
= \frac{1}{n}\left|
\|\bm{X}_{(n)}^\top \bm{v}\|_2^2
- E[\|\bm{X}_{(n)}^\top \bm{v}\|_2^2]
\right|.
\]
It follows from Lemma \ref{Hanson-Wright}
that there exists a constant $c>0$ such that
for every $\zeta>0$,
\[
P\left(
\frac{1}{n}\left|
\|\bm{X}_{(n)}^\top \bm{v}\|_2^2
- E[\|\bm{X}_{(n)}^\top \bm{v}\|_2^2]
\right|>\zeta \|Q_n\|_2
\right)
\leq 2 \exp(-cn \min\{\zeta, \zeta^2\}),
\]
where $Q_n$ is the covariance matrix of 
$\bm{X}_{(n)}^\top \bm{v}$.
Noting that 
\[
\|Q_n\|_2 \leq \sum_{l=0}^n \rho(l) \phi_{\max},
\]
we have 
\[
P\left(
\frac{1}{n}\left|
\|\bm{X}_{(n)}^\top \bm{v}\|_2^2
- E[\|\bm{X}_{(n)}^\top \bm{v}\|_2^2]
\right|>\zeta \sum_{l=0}^n \rho(l) \phi_{\max}
\right)
\leq 2 \exp(-cn \min\{\zeta, \zeta^2\}).
\]
Then, we take a union bound over 
$\mathcal{B}_0(2 \tilde{s})\cap\mathbb{S}^{p-1}$.
It is easy to see that we need 
\[
\left(\begin{array}{c}
p \\
2 \tilde{s}\\
\end{array}
\right) 6^{2\tilde{s}}
\]
points to cover the set $\mathcal{B}_0(2\tilde{s})\cap\mathbb{S}^{p-1}$ by balls with radius 
$1/2$.
See $e.g.$, Chapter 4 of \cite{vershynin2018high}.
We wright $\mathcal{N}_{1/2}$ for the 
set of centers of $1/2$-balls which covers 
$\mathcal{B}_0(2 \tilde{s})\cap\mathbb{S}^{p-1}$.
Then, it holds that 
\begin{eqnarray*}
\lefteqn{
P\left(\sup_{v \in \mathcal{B}_0(2 \tilde{s})\cap \mathbb{S}^{p-1}}
\frac{1}{n}\left|
\|\bm{X}_{(n)}^\top \bm{v}\|_2^2
- E[\|\bm{X}_{(n)}^\top \bm{v}\|_2^2]
\right|>\zeta \sum_{l=0}^n \rho(l) \phi_{\max}
\right)}\\
&\leq& 
\sum_{v \in \mathcal{N}_{1/2}}
P\left(
\frac{1}{n}\left|
\|\bm{X}_{(n)}^\top \bm{v}\|_2^2
- E[\|\bm{X}_{(n)}^\top \bm{v}\|_2^2]
\right|>\zeta \sum_{l=0}^n \rho(l) \phi_{\max}
\right) \\
&\leq& \left(\begin{array}{c}
p \\
2\tilde{s}\\
\end{array}
\right) 6^{2 \tilde{s}} 2 \exp(-cn \min\{\zeta, \zeta^2\})\\
&\leq& (6p)^{2 \tilde{s} }2 \exp(-cn \min\{\zeta, \zeta^2\})\\
&=& 2\exp\left(
2 \tilde{s} \log p + 2 \tilde{s} \log 6 - cn \min\{\zeta, \zeta^2\}
\right)\\
&\leq& 2\exp\left(
4 \tilde{s} \log p - cn \min\{\zeta, \zeta^2\}
\right) 
\end{eqnarray*}
If we replace $\zeta$ with $\tilde{\zeta}_n$, then we obtain the conclusion.
\end{proof}
\begin{proof}[Proof of Theorem \ref{oracle l0}]
It suffices to show the inequality
\[
\|\hat{\bm{\beta}}^0_n - \bm{\beta}^0\|_2
\leq \delta_n
\]
on the event
\[
\left\{\sup_{v \in \mathcal{B}_0(2 \tilde{s}) \cap \mathbb{S}^{p-1}} |\bm{v}^\top W_n \bm{v}| \leq \tilde{\gamma}_n\right\}.
\]
Note that 
$\hat{\bm{\beta}}^0_n- \bm{\beta}^0 \in \mathcal{B}_0(2 \tilde{s})$.
Along the same lines as the proof of Theorem \ref{oracle ineq Lasso}, 
we have the following analog to \eqref{eq:7.10ll}, \eqref{eq:7.7ll} and \eqref{eq:7.9ll} that
\begin{align*}
(\sigma-3\eta) \|\hat{\bm{\beta}}^0_n - \bm{\beta}^0\|_2^2 + \lambda_0 \|\hat{\bm{\beta}}^0_n\|_0
&\leq
R(\hat{\bm{\beta}}^0_n)-R(\bm{\beta}^0) + \lambda_0 \|\hat{\bm{\beta}}^0_n\|_0 \notag\\
&\leq
\bigl(\dot{R}_n(\hat{\bm{\beta}}^0_n)-\dot{R}(\hat{\bm{\beta}}^0_n) \bigr)^\top (\bm{\beta}^0 - \hat{\bm{\beta}}^0_n) + \lambda_0 \|\bm{\beta}^0\|_0 \notag\\
&\leq
\tilde{\gamma}_n \|\hat{\bm{\beta}}^0_n - \bm{\beta}^0\|_2^2+
\bm{\beta}^{0\top} W_n (\hat{\bm{\beta}}^0_n - \bm{\beta}^0) + \lambda_0 \|\bm{\beta}^0\|_0, 
\end{align*}
and thus, it holds that
\begin{align*}
(\sigma-3\eta - \tilde{\gamma}_n)\|\hat{\bm{\beta}}^0_n - \bm{\beta}^0\|_2^2
&\leq 
|\bm{\beta}^{0 \top}W_n(\hat{\bm{\beta}}_n^0 -\bm{\beta}^0)|
+ \lambda_0 \|\bm{\beta}^0\|_0
- \lambda_0 \|\hat{\bm{\beta}}_n^0\|_0 \\
&\leq \|W_n^{1/2}\bm{\beta}^0\|_2
\|W_n^{1/2} (\hat{\bm{\beta}}_n^0 -\bm{\beta}^0)\|_2
+ \lambda_0 \tilde{s} \\
&\leq \tilde{\gamma}_n\|\bm{\beta}^0\|_2
\|\hat{\bm{\beta}}_n^0 -\bm{\beta}^0\|_2
+ \lambda_0 \tilde{s}.
\end{align*}
Noting that $\|\bm{\beta}^0\|_2=\phi_{\max}$, we reach the conclusion 
after solving the quadratic inequality 
with respect to $\|\hat{\bm{\beta}}^0_n-\bm{\beta}^0\|_2$.
\end{proof}

\section*{Acknowledgements.}
The last two authors  would like to express their thanks to the Institute for Mathematical Science (IMS) and Research Institute for Science \& Engineering, Waseda University, respectively, for their support.

K. Fujimori is supported by JSPS Grant-in-Aid for Early-Career Scientists 21K13271.
Y. Liu is supported by JSPS Grant-in-Aid for Scientific Research (C) 20K11719.
M. Taniguchi is supported by JSPS Grant-in-Aid for Scientific Research (S) 18H05290.



\bibliographystyle{econ}
\bibliography{spca-ref}

\end{document}